%% file: main.tex
\documentclass[11pt,a4paper,oneside]{amsart}
\input{top.tex}

\title{Centralisers and the virtually cyclic dimension of \texorpdfstring{$\Out(F_N)$}{Out(FN)}}

\usepackage[foot]{amsaddr}
\author{Yassine Guerch}
\address[Yassine Guerch]{Univ. Lyon, ENS de Lyon, UMPA UMR 5669, 46 allée d'Italie,
F-69364 Lyon cedex 07, France}
\email{yassine.guerch@ens-lyon.fr}
\author{Sam Hughes} 
\address[Sam Hughes]{Mathematical Institute, Andrew Wiles Building, Observatory Quarter, University of Oxford, Oxford OX2 6GG, United Kingdom}
\email{sam.hughes@maths.ox.ac.uk}
\author{Luis Jorge Sánchez Saldaña}
\address[Luis Jorge S\'anchez Salda\~na]{Facultad de Ciencias, Universidad Nacional Aut\'onoma de M\'exico}
\email{luisjorge@ciencias.unam.mx}
\date{\today}


\begin{document}

\maketitle

\begin{abstract}
    We prove that the virtually cyclic (geometric) dimension of the finite index congruence subgroup $\mathrm{IA}_N(3)$ of $\mathrm{Out}(F_N)$ is $2N-2$.  From this we deduce the virtually cyclic dimension of $\mathrm{Out}(F_N)$ is finite.  Along the way we prove L\"uck's property (C) holds for $\mathrm{Out}(F_N)$, we prove that the commensurator of a cyclic subgroup of $\mathrm{IA}_N(3)$ equals its centraliser, we give an $\mathrm{IA}_N(3)$ analogue of various exact sequences arising from reduction systems for mapping class groups, and give a near complete description of centralisers of infinite order elements in $\mathrm{IA}_3(3)$.
\end{abstract}

\section{Introduction}

Let $F_N$ denote the free group on $N$ generators and let $\Out(F_N)$ denote its outer automorphism group.  The study of $\Out(F_N)$ has been ubiquitous in geometric group theory and low dimensional topology; finding connections with arithmetic groups, mapping class groups, moduli spaces of graphs, and many others.  Despite this, the topology of $\Out(F_N)$ has been notoriously hard to pin down.  

It is known that the virtual cohomological dimension of $\Out(F_N)$ is $2N-3$ \cite{CV86} and that $\Out(F_N)$ is a virtual duality group \cite{BestvinaFeighn2000}. Much work has gone into computing various Euler characteristics of $\Out(F_N)$ \cite{SmillieVogtmann1987,SmillieVogtmann1987b,BorinskyVogtmann2020} and in low dimensions the rational cohomology has been computed \cite{HatcherVogtmann1998,Ohashi2008,Bartholdi2016}. Moreover, some homological stability phenomena has been observed \cite{Hatcher1995,HatcherVogtmann2004,HatcherVogtmannWahl2006}.  In this article we will be concerned with a topological property of a different flavour.

\subsection*{Virtually cyclic dimension}
Given a group $G$, a collection of subgroups $\calg$ is called a \emph{family} if it is closed under conjugation and under taking subgroups. We say that a $G$-CW-complex $X$ is a model for the classifying space $E_\calg G$ if, given any $G$-CW-complex $Y$ with isotropy in $\calg$, there is up to homotopy a unique $G$-map $Y\to X$. Such a model always exists and it is unique up to $G$-homotopy equivalence. The \emph{geometric dimension} of $G$ with respect to the family $\calg$, denoted $\gd_\calg (G)$, is the minimum dimension $n$ such that $G$ admits an $n$-dimensional model for $E_\calg G$.

The most studied families are: $\mathcal{TR}$ the family that contains only the trivial subgroup, $\FIN$ the family of finite subgroups, and $\mathcal{VC}$ the family of virtually cyclic subgroups. For the first, one recovers $EG$ and all of the classical group cohomology to go with it.  The latter two are relevant to the isomorphism conjectures in $K$-theory, $E_\FIN G$, denoted $\underbar{E}G$, is relevant to the Baum--Connes Conjecture, and $E_\mathcal{VC} G$, the topic of this article and which we will henceforth denote by, $\underline{\underline{E}}G$, is relevant to the Farrell--Jones Conjecture.

The Farrell--Jones Conjecture, one of the most prominent conjectures in modern topology, predicts that a certain `assembly map'  
\[H_n^G(\mathrm{pr})\colon H^G_n(\underline{\underline{E}} G; \mathbf{K}_R)\to K_n(RG) \]
is an isomorphism.  Whilst we will not explain all of the ramifications and developments of the Farrell--Jones Conjecture and instead refer the reader to the book project of L\"uck \cite{LuckBookProj}.  We do point out that the Farrell--Jones Conjecture is still open for $\Out(F_N)$ but is known for mapping class groups of finite type surfaces \cite{BartelsBestvina2019}.   Clearly, the left hand side is explicitly concerned with the classifying space for virtually cyclic actions and so understanding the minimal possible dimension for a model of $\underline{\underline{E}} G$ is of great importance.

In the present article we study the geometric $\mathcal{VC}$-dimension, which from now on we denote by $\underline{\underline{\gd}}(G)$ for a group $G$.  Computations for $\underline{\underline{\gd}}$ are known for (relatively) hyperbolic groups \cite{JuanPiLeary2006,LafontOrtiz2007}, elementary amenable groups \cite{DegrijsePetrosyan2014}, discrete linear groups \cite{DegrijseKohlPetrosyan2015}, $\CAT(0)$ groups \cite{luck2009classifying,DegrijsePetrosyan2015}, virtually polycyclic groups \cite{LW12,CFH06}, mapping class groups of finite type surfaces \cite{JuanPiTru2016,Nucinkis:Petrosyan}, mapping class groups of punctured spheres \cite{AramonyaJuanTruj2018},  systolic groups \cite{OsajdaPrytula2018}, braid groups \cite{FloresGonzalez2020}, and orientable $3$-manifold groups \cite{JoeckenLafontSS2021}.

Our goal is to establish that $\underline{\underline{\gd}}(\Out(F_N))$ is finite. For our purposes we are interested also in the finite index congruence subgroup \[\IA_N(3)\coloneqq\ker(\Out(F_N) \to \Aut(H_1(F_N;\FF_3)).\] Our main result is the following:

\medskip
\begin{duplicate}[\Cref{thmx.gdIA3}]
    Let $N\geq 1$.  Then, $\underline{\underline{\gd}}(\IA_N(3))=2N-2$.
\end{duplicate}
\medskip

For a group $G$ we denote by $\underline{\gd}(G)$ the minimal dimension of a model for $\underline{E}G$.  A problem of L\"uck \cite[Problem 10.51]{LuckBookProj} asks for which groups $G$ do the inequalities $\underline{\gd}(G)-1\leq\underline{\underline{\gd}}(G)\leq \underline{\gd}(G)+1$ hold. The previous theorem answers this in the affirmative for all finite index subgroups of $\IA_N(3)$.  From here we establish our desired result:

\medskip
\begin{duplicate}[\Cref{corx.gdOutFn}]
    Let $N\geq 1$.  Then, $\underline{\underline{\gd}}(\Out(F_N))$ is finite.
\end{duplicate}
\medskip

In a sense, our proof is similar to the analogous results for mapping class groups of finite type surfaces established by Juan-Pineda--Trujillo-Negrete \cite{JuanPiTru2016} and Nucinkis--Petrosyan \cite{Nucinkis:Petrosyan}.  The proofs in the mapping class group case heavily rely on reduction systems and the Nielsen--Thurston classification of mapping classes.  Two tools that are not readily available in the $\Out(F_N)$ setting.  

One of the key steps in both proofs for mapping class groups is to the use L\"uck--Weiermann push-out construction \cite{LW12}.  This requires a description of the commensurators of infinite cyclic subgroups by means of short exact sequences coming from the reduction systems of various elements. This process allows for an inductive argument. To this end Juan-Pineda and Trujillo-Negrete prove that the commensurator of any infinite cyclic subgroup $C$ of $\mathrm{MCG}(S)$ can be realized as the normaliser of a finite index subgroup of $C$. We obtain analogous results for $\Out(F_N)$ that we now describe in great detail. 

\medskip

\subsection*{Centralisers and commensurators}

Towards proving the result of Juan-Pineda and Trujillo-Negrete for $\Out(F_N)$ we establish L\"uck's \emph{Property (C)} (see \Cref{sec:PropC} for a definition).  We then use this and results in \cite{Guerch2022Roots} to deduce the following theorem.

\medskip

\begin{duplicate}[\Cref{Coro comm centraliser}]
Let $\phi \in \IA_N(3)$. The commensurator of the cyclic group $\langle\phi\rangle$ in $\IA_N(3)$ is equal to its centraliser.    
\end{duplicate}

\medskip

Centralisers of elements of $\Out(F_N)$ have been widely studied in the literature (see for instance~\cite{bestvina1997laminations,kapovich2011stabilizers,algom2017normalizers,rodenhausen2015centralisers,AndrewMartino2022,mutanguha2022limit,Guerch2022Roots}). The centralisers of large families of elements of $\Out(F_N)$ are now completely understood, for example, fully irreducible outer automorphisms by Bestvina--Feighn--Handel~\cite{bestvina1997laminations}, atoroidal elements by the work of Feighn--Handel~\cite{feighn2009abelian} or linearly growing elements by Rodenhausen--Wade~\cite{rodenhausen2015centralisers} and Andrew--Martino~\cite{AndrewMartino2022}. All these families are, in some sense, analogues of either pseudo-Anosov homeomorphisms or Dehn twists homeomorphisms and their centralisers have a similar structure. In all these cases, the proofs also imply that the centralisers have a finite index subgroup with a finite classifying space.

However, one major difficulty in understanding the centraliser of an arbitrary element of $\Out(F_N)$ is the lack of a complete analogue of the reduction system as in the case of the mapping class groups. This often prevents us from understanding the centraliser of an automorphism using its action on free groups of smaller ranks in analogy with understanding a mapping class by its action on subsurfaces.

The theory of reduction systems for $\Out(F_N)$ has a long history as its use in the understanding of mapping class groups is central. However, contrary to the case of surface homeomorphisms, the study of $\Out(F_N)$ often requires distinct constructions of reduction systems according to the context. The first incarnation dates back to the work of Bestvina--Handel on the existence of train tracks for automorphisms of free groups (\cite{BesHan92}, see also~\cite{BestvinaFeighnHandel00,feighn2011recognition}) in order to understand dynamical properties of an individual automorphism. Reduction systems for \emph{subgroups} of $\Out(F_N)$ first appear in the work of Bestvina--Feighn--Handel on the Tits alternative for $\Out(F_N)$~\cite{BestvinaFeighnHandel00,BestvinaFeighnHandel04,BestvinaFeighnHandel05}. The construction of invariant free factors of subgroups of $\Out(F_N)$ plays a key role in the work of Handel--Mosher~\cite{HandelMosher20}, which led to the study of the bounded cohomology of $\Out(F_N)$~\cite{handel2015hyperbolic,handel2017hyperbolic}. Finally, a notion of a dynamical reduction system for subgroups of $\Out(F_N)$ was also constructed by Guirardel--Horbez in order to prove a measure equivalence rigidity result for $\Out(F_N)$~\cite{guirardel2021measure}.

In the present paper, we need to understand classifying spaces of centralisers of \emph{arbitrary} elements of $\Out(F_N)$ in order to apply L\"uck--Weiermann push-out construction. Therefore, one of the main steps in the proof of \Cref{thmx.gdIA3} is to give a another analogue of the reduction systems for mapping class groups which is well-adapted to the study of centralisers of elements of $\IA_N(3)$.  This is the content of the following theorem. We refer to \Cref{Section preliminaries} for definitions and notations.

\medskip

\begin{duplicate} [\Cref{thm:classificationcentralisers}]
    Let $N \geq 2$ and let $\phi \in \IA_N(3)$. The centraliser $C(\phi)$ of $\phi$ in $\IA_N(3)$ satisfies one of the followings.
    \begin{enumerate}
        \item The outer automorphism $\phi$ is a Dehn twist. There exist a JSJ tree $T$ preserved by $C(\phi)$ and a short exact sequence 
    \[1 \to K \to C(\phi) \to \prod_{v\in V(F_N \backslash T)} \IA_v(3) \to 1, \] where $K$ is a free abelian group whose dimension is equal to $|E(F_N \backslash T)|$ and, for every $v\in V(F_N \backslash T)$, the group $\IA_v(3)$ is a finite index subgroup of the group $\Out(G_v,\mathrm{Inc}_v)$. Moreover, $\phi$ is contained in $K$.
        \item There exist $A_1,A_2 \subseteq F_N$ with $F_N=A_1 \ast A_2$, $\rank(A_1),\rank(A_2) \leq N-1$ and a homomorphism
        \[C(\phi) \to \IA(A_1,3) \times \IA(A_2,3)\] 
         whose kernel is a finite index subgroup of a direct product of two finitely generated free (maybe trivial or cyclic) groups.
        \item There exist $A_1,\ldots, A_k,B \subseteq F_N$ nontrivial with $F_N=A_1 \ast \ldots A_k \ast B$ and a homomorphism \[C(\phi) \to \ZZ \times \prod_{i=1}^k \IA(A_i,3)\] whose kernel is abelian and $\phi$ projects onto the $\ZZ$ factor.
        \item There exist a JSJ tree $T$ preserved by $C(\phi)$, a partition $VT=V_1\coprod V_2$ and a homomorphism \[C(\phi) \to \ZZ \times \prod_{v \in V_2} \Out(G_v)\] whose kernel is abelian and $\phi$ projects onto the $\ZZ$ factor.
        \end{enumerate}
\end{duplicate}
\medskip

We note that the case of a Dehn twist was proved by Rodenhausen--Wade~\cite{rodenhausen2015centralisers} (see also the work of Cohen--Lustig~\cite{lustig1999conjugacy}). \Cref{thm:classificationcentralisers} allows us to understand the centraliser of an element of $\IA_N(3)$ through its action on free groups of smaller ranks. Note that the main limitation when compared to reduction theory of mapping class groups is that the homomorphisms given in \Cref{thm:classificationcentralisers} are not necessarily surjective. In particular, we do not always have control on the image. We expect that \Cref{thm:classificationcentralisers} is a significant step towards a complete understanding of centralisers of \emph{all} elements of $\Out(F_N)$ and will be of independent interest. We remark that the reduction systems for $\Out(F_N)$ described above are all distinct from ours but, to our knowledge, the aforementioned ones could not be directly used to compute the geometric dimension of centralisers and Weyl groups.

Using \Cref{thm:classificationcentralisers} we are able to prove an inductive structural description of the Weyl group of an infinite order element in $\IA_N(3)$ (see \Cref{Coro:sesWeylgroup}).  Our other major result on centralisers is a near comprehensive structure description of centralisers of infinite order elements in $\IA_3(3)$.  This result is in a sense a specialisation of \Cref{thm:classificationcentralisers} but requires a careful and explicit analysis of stabilisers of free factor systems of $F_3$.

\medskip

\begin{duplicate}[\Cref{Theo type VF}]
Let $\phi \in \IA_3(3)$. The centraliser $C(\phi)$ of $\phi$ in $\IA_3(3)$ is of type $\mathsf{VF}$. Moreover, one of the following holds.
\begin{enumerate}
\item The centraliser of $\phi$ is abelian.
\item The centraliser of $\phi$ is isomorphic to $F \times \ZZ$ where $F$ is a finitely generated free group.
\item The centraliser of $\phi$ is isomorphic to a direct product $H\times \ZZ$ where $H$ is a finite index subgroup of a direct product of two finitely generated free groups.
\item The outer automorphism $\phi$ is a Dehn twist. There exist a JSJ tree $T$ preserved by $C(\phi)$ and a short exact sequence 
    \[1 \to K \to C(\phi) \to \prod_{v\in V(F_3 \backslash T)} \IA_v(3) \to 1, \] where $K$ is a free abelian group whose dimension is equal to $|E(F_3 \backslash T)|$ and, for every $v\in V(F_3 \backslash T)$, the group $\IA_v(3)$ is a finite index subgroup of the group $\Out(G_v,\mathrm{Inc}_v)$. Moreover, $\phi$ is contained in $K$.
\end{enumerate}
\end{duplicate}

\medskip

We note that in \Cref{Theo type VF}, the conclusion that the centralisers are of type $\mathsf{VF}$ improves a result of Francaviglia--Martino--Syrigos~\cite[Theorem~8.2.1]{francaviglia2021action}; where they prove that centralisers of infinite order elements of $\Out(F_3)$ are finitely generated (type $\mathsf{F}_1$). This finiteness result crucially depends on work of Rodenhausen--Wade \cite{rodenhausen2015centralisers} where they prove centralisers of Dehn twists are of type $\mathsf{VF}$. 

\medskip

\begin{duplicate}[\Cref{cor type VF}]
    Let $\phi\in\IA_3(3)$.  The centraliser $C_{\Out(F_3)}(\phi)$ is of type $\mathsf{VF}$.
\end{duplicate}

\subsection*{Structure of the paper}
In \Cref{Section preliminaries} we recall the necessary background on free factor systems, $\IA_N(3)$, and relative free factor graphs.  

In \Cref{sec:PropC} we establish L\"uck's property (C) for $\Out(F_N)$. We then use property (C) and results of \cite{Guerch2022Roots} to prove \Cref{Coro comm centraliser}.

In \Cref{sec:centralisers} we prove \Cref{thm:classificationcentralisers}. The proof of this theorem is separated into two distinct subsections, according to the fixed subgroups of the considered outer automorphism. If the fixed subgroups fill the group $F_N$, then one can apply the theory of JSJ decompositions of groups. This is done in \Cref{sec:oneended}. Otherwise, the main input is work of Horbez--Wade \cite{HorbezWade20} and Guirardel--Horbez (where they attribute some work to Guirardel and Levitt) \cite{Guirardelhorbez19} as well as a very careful study of the free factor systems that arise. This is done in \Cref{sec:infinitelyended}. From here we deduce \Cref{Coro:sesWeylgroup} which gives an inductive description of the Weyl groups of elements in $\IA_N(3)$.  The work here is a key step towards being able to inductively apply the L\"uck--Weiermann construction.  But, we suspect the results established here will be of independent interest.

In \Cref{sec:OutF3} we begin to specialise \Cref{thm:classificationcentralisers} to $\IA_3(3)$ which will form the base case of our induction later. The main result of this section is \Cref{Theo type VF}. A thorough analysis of elements of $\IA_3(3)$ allows us to refine \Cref{thm:classificationcentralisers}.

In \Cref{sec:GeomDimDehn} we study the \emph{proper geometric dimension} $\underline{\gd}(W(\phi))$, where $W(\phi)$ is the Weyl group of a Dehn twist $\phi$ in $\IA_N(3)$, proving it is at most $2N-4$.  Here $\underline{\gd}$ denotes the minimal dimension of a model for $\underline{E}G$.  The need for this apparent diversion is that upper bounds on the proper geometric dimension feed into our inductive argument.  The remaining cases of possible centralisers and their Weyl groups are dealt with later but the argument for Dehn twists turns out to be somewhat more technical.

Finally, in \Cref{sec:ThmA} we combine our analysis of the centralisers, commensurators, and Weyl groups of $\IA_N(3)$ to prove \Cref{thmx.gdIA3} and deduce \Cref{corx.gdOutFn}.  The build up to the use of L\"uck--Weiermann turns out to be involved.  The key steps being \Cref{thm.gdIA3} and \Cref{prop:WeylN} where we study $\underline{\gd}(W(H))$ for the Weyl group of $H$ an infinite cyclic subgroup of $\IA_3(3)$, or a maximal infinite cyclic subgroup of $\IA_N(3)$ respectively.   At this point the main theorem is at hand.

\subsection*{Acknowledgements}
The first author was supported by the LABEX MILYON of Université de Lyon.  The second author received funding from the European Research Council (ERC) under the European Union's Horizon 2020 research and innovation programme (Grant agreement No. 850930). The third author is grateful for the financial support of DGAPA-UNAM grant PAPIIT IA106923.  All three authors thank Naomi Andrew, Ian J. Leary, Armando Martino, and Ric Wade for helpful comments and corrections on an earlier draft of this manuscript.

\section{Preliminaries}\label{Section preliminaries}

\subsection{Free factor systems} 

Let $N \geq 2$ and let $F_N$ be a nonabelian free group of rank $N$. A \emph{free factor system of $F_N$} is a finite set $\mathcal{F}=\{[A_1],\ldots,[A_k]\}$ of conjugacy classes of subgroups of $F_N$ such that there exists a subgroup $B$ of $F_N$ with $F_N=A_1 \ast \ldots \ast A_k \ast B$. 

A free factor system $\mathcal{F}$ of $F_N$ is \emph{sporadic} if either $\mathcal{F}=\{[A]\}$ and $F_N=A \ast \ZZ$ or $\mathcal{F}=\{[A],[B]\}$ and $F_N=A \ast B$. Otherwise, we say that $\calf$ is \emph{nonsporadic}.

The group $\Out(F_N)$ has a natural action on the set of free factor systems and we denote by $\Out(F_N,\mathcal{F})$ the stabiliser of a free factor system $\mathcal{F}$. Let $\phi \in \Out(F_N)$ and let $\mathcal{F}$ be a free factor system of $F_N$. Suppose that $\phi$ fixes every element of $\mathcal{F}$. Then, for every $[A] \in \mathcal{F}$, by malnormality of $A$, the element $\phi$ induces an element $\phi|_A \in \Out(A)$.

The collection of free factor systems is equiped with a natural partial order, where $\mathcal{F}_1 \leq \mathcal{F}_2$ if for every $[A] \in \mathcal{F}_1$, there exists $[B] \in \mathcal{F}_2$ such that $A \subseteq B$.

\subsection{\texorpdfstring{Properties of the subgroup $\mathrm{IA}_N(3)$}{Properties of the subgroup IAN(3)}}

Let $N \geq 2$ and let 
\[
\IA_N(3)=\ker(\Out(F_N) \to \Aut(H_1(F_N;\FF_3)).
\]
In this section, we recall some properties of $\IA_N(3)$. Most of them show some aperiodic properties of the group $\IA_N(3)$ which will be of great interest in the rest of the paper.

\begin{prop}\cite{baumslag1967centre}\label{Prop torsion free}
The group $\IA_N(3)$ is torsion free.
\end{prop}

\begin{thm}\label{Theo:freefactorconjclassfixed}\cite[Theorem~II.3.1]{HandelMosher20} 
Let $H$ be a subgroup of $\IA_N(3)$.

\medskip

\noindent{$(1)$ } Suppose $\mathcal{F}$ is an $H$-periodic free factor system. Then $\mathcal{F}$ is fixed by $H$ and every element $[A] \in \mathcal{F}$ is fixed by $H$.

\medskip

\noindent{$(2)$ } If $[g]$ is an $H$-periodic conjugacy class of some element $g \in F_N$, then $[g]$ is fixed by $H$.
\end{thm}

\begin{thm}\cite[Theorem~1.1]{HandelMosher20abelian}\label{Theo virtually abelian is abelian}
Let $H$ be a virtually abelian subgroup of $\IA_N(3)$. Then $H$ is abelian and finitely generated.
\end{thm}

\begin{lemma}\cite[Lemma~2.9]{Guerch2022Roots}\label{Lem Tits}
A subgroup $H$ of $\IA_N(3)$ is abelian if and only if it does not contain a nonabelian free group.
\end{lemma}

\begin{thm}\cite[Theorem~1.1]{Guerch2022Roots}\label{Theo:roots}
For all $\phi, \psi \in \IA_N(3)$, if there exists $m \in \ZZ^*$ such that $\phi^m=\psi^m$, then $\phi=\psi$.
\end{thm}

\begin{prop}\cite[Corollary~4.2]{Guerch2022Roots}\label{Prop:powercommute}
For all $\phi, \psi \in \IA_N(3)$, if there exist $m,n \in \ZZ^*$ such that $\phi^m$ and $\psi^n$ commute, then $\phi$ and $\psi$ commute.
\end{prop}

\subsection{The relative free factor graph}

In this section, we introduce a Gromov hyperbolic space on which $\Out(F_N)$ acts by isometries. This space will play a key role in the study of centralisers of elements of $\Out(F_N)$ as centralisers will fix points of its Gromov boundary. 

Let $\mathcal{F}$ be a free factor system of $F_N$. An \emph{$(F_N,\mathcal{F})$-free factor system} is a proper free factor system $\mathcal{F}'$ of $F_N$ with $\mathcal{F}<\mathcal{F}'$.  An \emph{$(F_N,\mathcal{F})$-free factor} is a subgroup $A$ of $F_N$ such that there exists an $(F_N,\mathcal{F})$-free factor system $\mathcal{F}'$ of $F_N$ with $[A] \in \mathcal{F}'$. 

The \emph{free factor graph of $F_N$ relative to $\mathcal{F}$}, denoted by $\FF(F_N,\mathcal{F})$, is the graph whose vertices are the conjugacy classes of $(F_N,\mathcal{F})$-free factors of $F_N$, two such conjugacy classes $[A], [B]$ being adjacent if either $A \subsetneq B$ or $B \subsetneq A$.

By a result of Handel--Mosher~\cite{HandelMosher14}, the graph $\FF(F_N,\mathcal{F})$ is Gromov-hyperbolic (see also the work of Bestvina--Feighn~\cite{bestvina2014hyperbolicity} for the case $\mathcal{F}=\varnothing$ and the work of Guirardel--Horbez~\cite[Proposition~2.11]{Guirardelhorbez19} for general free products of groups). 

The group $\Out(F_N,\mathcal{F})$ acts naturally on $\FF(F_N,\mathcal{F})$ by isometries. An outer automorphism $\phi \in \Out(F_N,\mathcal{F})$ is \emph{fully irreducible relative to $\mathcal{F}$} if there does not exist a proper free factor system $\mathcal{F} < \mathcal{F}'$ fixed by a power of $\phi$. These elements are the loxodromic elements of $\FF(F_N,\mathcal{F})$. 

\begin{thm}\cite[Theorem~A]{gupta18}\label{Theo loxo free factor}
Let $\mathcal{F}$ be a nonsporadic free factor system of $F_N$. An element $\phi \in \Out(F_N,\mathcal{F})$ is a loxodromic element of $\FF(F_N,\mathcal{F})$ if and only if $\phi$ is fully irreducible relative to $\mathcal{F}$.
\end{thm}

The following theorem was proved by Handel and Mosher~\cite{HandelMosher20} when the subgroup is finitely generated case and by Guirardel and Horbez~\cite{Guirardelhorbez19} in the general case.

\begin{thm}\cite[Theorem~7.1]{Guirardelhorbez19}\cite[Theorem~A]{HandelMosher20}\label{Theo fully irreducible contained}
Let $H$ be a subgroup of $\IA_N(3)$ and let $\mathcal{F}$ be a maximal proper $H$-invariant free factor system.  Suppose that $\mathcal{F}$ is nonsporadic. Then $H$ contains a fully irreducible outer automorphism relative to $\mathcal{F}$.
\end{thm}

We record the following fact, which is a consequence of the description of the Gromov boundary of $\FF(F_N,\mathcal{F})$. It is due to Hamenstädt~\cite{Hamenstadt14} for the case $\mathcal{F}=\varnothing$, and Guirardel and Horbez~\cite{Guirardelhorbez19} for the general case. We refer to \cite[Section~3]{Guirardelhorbez19} for the definition of an $(F_N,\mathcal{F})$-arational tree. 

\begin{prop}\cite[Theorem~3.4]{Guirardelhorbez19}\label{prop fix arational boundary}
Let $\mathcal{F}$ be a nonsporadic free factor system of $F_N$ and let $H$ be a subgroup of $\Out(F_N,\mathcal{F})$. If $H$ has a finite orbit in $\partial_{\infty} \FF(F_N,\mathcal{F})$, then $H$ has a finite index subgroup which fixes the homothety class of an $(F_N,\mathcal{F})$-arational tree.
\end{prop}

For the rest of the article, we only need to know some properties of the stabiliser in $\Out(F_N,\mathcal{F})$ of the homothety class $[T]$ of an $(F_N,\mathcal{F})$-arational tree $T$. We have a natural homomorphism \[\mathrm{SF}\colon \Stab([T]) \to \RR_+^\times\] 
given by the stretching factor, whose kernel is denoted by $\Stab_{\Isom}(T)$. The homomorphism $\mathrm{SF}$ has the following properties. 

\begin{lemma}\cite[Lemma~6.2, Proposition~6.3, Corollary~6.12]{Guirardelhorbez19}\label{Lem stratching factor cyclic} The following hold:
\begin{enumerate}
    \item The image of $\mathrm{SF}$ is cyclic. 
    \item For every $\phi \in \Stab([T])$, we have $\mathrm{SF}(\phi) \neq 1$ if and only if $\phi$ is fully irreducible relative to $\mathcal{F}$.
\end{enumerate}
\end{lemma}  

\subsection{JSJ decompositions of free groups}\label{Section:JSJ}

This section follows the work of Guirardel--Levitt~\cite{guirardellevitt2017jsj}. 

An \emph{$F_N$-tree} is a simplicial tree equipped with an action of $F_N$ by isometries. Let $\cala$ be a finite set of conjugacy classes of finitely generated subgroups of $F_N$. We say that $F_N$ is \emph{one-ended relative to $\cala$} if there does not exist an $F_N$-tree $T$ with trivial edge stabilisers such that, for every $[A]\in \cala$, the group $A$ fixes a point in $T$. Otherwise, the group $F_N$ is \emph{infinitely-ended relative to $\cala$}.

If $F_N$ is one-ended relative to $\cala$, by~\cite[Theorem~9.14]{guirardellevitt2017jsj}, there exists an $F_N$-tree $T_\cala$ with infinite cyclic edge stabilisers called the \emph{JSJ tree relative to $\cala$}. We record some of its properties in the rest of the section. 

The group $\Aut(F_N)$ acts on the set of $F_N$-equivariant isometry classes of $F_N$-trees by precomposition of the action, and this action passes to the quotient to give an action of $\Out(F_N)$ on the set of $F_N$-equivariant isometry classes of $F_N$-trees. We now prove a lemma which describes the action of an element $\phi$ of $\IA_N(3)$ on a JSJ tree whose $F_N$-equivariant isometry class is preserved by $\phi$.

\begin{lemma}\label{Lem:IAactstriviallyquotientgraph}
    Let $\calt$ be the $F_N$-isometry class of a JSJ tree $T$ and let $\phi \in \Stab(\calt) \cap \IA_N(3)$. The graph automorphism of $\Gamma=F_N\backslash T$ induced by $\phi$ is the identity.
\end{lemma}

\begin{proof}
    We follow~\cite[Proposition~2.15]{guirardel2021measure}. First note that, following the terminology of Guirardel-Levitt~\cite{guirardel2011trees}, the tree $T$ is a \emph{collapsed tree of cylinders for commensurability} (see~\cite[Theorem~9.14]{guirardellevitt2017jsj}). In particular, if $v,w \in VT$ are two distinct vertices with cyclic stabilisers, then $G_v$ and $G_w$ are not commensurable. Moreover, two distinct edge stabilisers are commensurable if and only if they have a common endpoint whose stabiliser is infinite cyclic. 
    
    Suppose first that $\Gamma$ is a tree. Let $v,w \in V\Gamma$ be distinct leaves. If $G_v$ is cyclic, then by the above paragraph, $G_w$ is not commensurable with $G_v$. By~\Cref{Theo:freefactorconjclassfixed}~$(2)$, the vertex $v$ is fixed by $\phi$. 
    
    So we may suppose that $G_v$ and $G_w$ are not infinite cyclic. Let $Q$ be the quotient of $H_1(F_N;\FF_3)$ by the group generated by the edge stabilisers of $T$. As edge stabilisers in $T$ are infinite cyclic, the image $Q_v$ and $Q_w$ of $G_v$ and $G_w$ in $Q$ is not trivial and $Q_v \neq Q_w$. As $\phi$ acts as the identity on $Q$, it fixes both $Q_v$ and $Q_w$. Hence $\phi$ fixes $v$ and $w$. Thus, $\phi$ fixes every leaf of $\Gamma$. This shows that $\phi$ fixes pointwise the graph $\Gamma$. This concludes the proof when $\Gamma$ is a tree.

    Suppose that $\Gamma$ is not a tree and let $\Gamma_0 \subseteq \Gamma$ be the subgraph consisting of all the embedded loops in $\Gamma$. We first show that $\phi$ fixes pointwise $\Gamma_0$. Note that $\phi$ also acts as the identity on $H_1(\Gamma;\FF_3)$. Thus, if $\Gamma_0$ is not a circle, then $\phi$ fixes pointwise the graph $\Gamma_0$. 

    Thus, it remains to treat the case when $\Gamma_0$ is a circle. Note that $\phi$ acts as an orientation preserving homeomorphism of $\Gamma_0$ as it acts trivially on $H_1(\Gamma;\FF_3)$. Hence $\phi$ acts as a rotation on $\Gamma_0$. Thus, it suffices to show that $\phi$ fixes a point of $\Gamma_0$ in order to show that $\phi$ fixes pointwise $\Gamma_0$. 
    
    Suppose that there exists $v\in V\Gamma_0$ with infinite cyclic stabiliser. By the first paragraph, the stabiliser of any vertex of $\Gamma_0$ distinct from $v$ is not commensurable with $G_v$. Thus, by~\Cref{Theo:freefactorconjclassfixed}~$(2)$, the stabiliser $G_v$ is fixed elementwise by a representative of $\phi$ and the vertex $v$ is fixed by $\phi$.

    Suppose that no vertex of $\Gamma_0$ has infinite cyclic vertex stabiliser. By the first paragraph, two distinct edges in $\Gamma_0$ have non commensurable edge stabilisers. As $\phi$ fixes the conjugacy class of any edge stabiliser by~\Cref{Theo:freefactorconjclassfixed}~$(2)$, it follows that $\phi$ must fix every edge of $\Gamma_0$.

    In all cases, we see that $\phi$ fixes pointwise $\Gamma_0$. As $\phi$ fixes elementwise the set of leaves of $\Gamma$, this implies that $\phi$ fixes pointwise $\Gamma$.
\end{proof}

We denote by $\Out(F_N,\cala)$ the group of outer automorphisms of $F_N$ preserving $\cala$ and by $\Out(F_N,\cala^{(t)})$ the subgroup of $\Out(F_N,\cala)$ such that, for every $[A]\in \cala$, an element $\phi \in \Out(F_N,\cala^{(t)})$ has a representative fixing $A$ elementwise. We also denote by $\IA(\cala,3)$ (resp. $\IA(\cala^{(t)},3)$) the group $\Out(F_N,\cala) \cap \IA_N(3)$ (resp. $\Out(F_N,\cala^{(t)}) \cap \IA_N(3)$).

\begin{thm}\cite[Theorem~9.14]{guirardellevitt2017jsj}\label{Thm:JSJ}
  Let $\cala$ be a finite set of conjugacy classes of finitely generated subgroups of $F_N$ such that $F_N$ is one-ended relative to $\cala$. The tree $T_\cala$ satisfies the following properties.  

  \begin{enumerate}
      \item Edge stabilisers are infinite cyclic.
      \item For every $[A]\in \cala$, the group $A$ fixes a point in $T_\cala$.
      \item The group $\Out(F_N,\cala)$ preserves the $F_N$-equivariant isometry class of $T$.
      \item We have a partition $VT_\cala=V_1\coprod V_2$ of the vertices of $T_\cala$ such that:
      \begin{enumerate}
          \item for every $v\in V_1$, the group $G_v$ is isomorphic to the fundamental group of a compact hyperbolic surface $\Sigma_v$ with infinite mapping class group such that for every $e\in ET$ adjacent to $v$, the group $G_e$ is contained in a boundary subgroup; 
          \item for every $[A]\in \cala$ and every $v\in G_{v_1}$ the intersection $A \cap G_v$ is contained in a boundary subgroup:
          \item for every $v\in V_1$, the image of the homomorphism $\IA(\cala,3) \to \Out(G_v)$ is contained in $\MCG(\Sigma_v)$;
          \item for every $v\in V_2$, the image $\IA(\cala^{(t)},3) \to \Out(G_v)$ is trivial. In that case, we say that $v$ is \emph{rigid}.
      \end{enumerate}
  \end{enumerate}
\end{thm}

\subsection{Dehn twist outer automorphisms}\label{section:twists}

In this section, we describe some special types of outer automorphisms called \emph{Dehn twist outer automorphisms}. Dehn twist outer automorphisms were intensively studied (see for instance~\cite{cohen1995very,lustig1999conjugacy,levitt2005,rodenhausen2015centralisers}). 

Let $N \geq 2$. If $a \in F_N$, we denote by $\Aut(F_N,a)$ the subgroup of $\Aut(F_N)$ fixing $a$ and by $\Aut(F_N,[a])$ the subgroup of $\Aut(F_N)$ preserving the conjugacy class $[a]$ of $a$. Let $\Out(F_N,[a])$ be the image of $\Aut(F_N,[a])$ in $\Out(F_N)$. These groups are generally called \emph{McCool groups} in the literature~\cite{guirardel2016mccool,bestvina2020mccool}.

In order to define a Dehn twist outer automorphism, we use the JSJ decomposition described in the previous section. 

Let $\phi \in \IA_N(3)$ and let $\cala=\{[\Fix(\Phi)]\}_{\Phi \in \phi}$. Note that $\cala$ also contains conjugacy classes of fixed subgroups of automorphisms in the outer class $\phi$ which are cyclic. The set $\cala$ is a finite set of conjugacy classes of finitely generated subgroups of $F_N$ by~\cite{BesHan92,gaboriau1998index}. Note that $\cala$ is stabilised by the centraliser $C(\phi)$ of $\phi$ in $\IA_N(3)$.

The outer automorphism $\phi$ is a \emph{Dehn twist outer automorphism} if $F_N$ is one ended relative to $\cala$ and, for every vertex $v\in VT_\cala$ of the JSJ tree associated with $\cala$, the vertex $v$ is rigid. Using \Cref{Thm:JSJ}~$(d)$, we see in particular that, the homomorphism $\langle \phi \rangle \to \prod_{v\in V(F_N\backslash T_\cala)} \Out(G_v)$ is trivial. Thus, for every $v\in VT_\cala$, there exists $[A]\in \cala$ such that $G_v \subseteq A$. If $v\in V(F_N \backslash T_\cala)$, we denote by $\Out(G_v,\mathrm{Inc}_v)$ the group of outer automorphisms of $G_v$ preserving the conjugacy classes of the incident edge stabilisers.

Our definition of Dehn twists outer automorphisms is not standard but is equivalent to the usual one for elements of $\IA_N(3)$ (this is a consequence of for instance~\cite[Lemma~5.33]{feighn2019conjugacy}). 

Rodenhausen and Wade~\cite{rodenhausen2015centralisers} described the centraliser of a Dehn twist $\phi\in \IA_N(3)$ in terms of its action on $T_\cala$.

\begin{thm}\label{Thm:sesDehntwist}
    Let $\phi \in \IA_N(3)$ be a Dehn twist. Its centraliser $C(\phi)$ in $\IA_N(3)$ fits in a short exact sequence 
    \[1 \to K \to C(\phi) \to \prod_{v\in V(F_N \backslash T_\cala)} \IA_v(3) \to 1, \] 
    where $K$ is a free abelian group whose dimension is equal to $|E(F_N \backslash T_\cala)|$ and, for every $v\in V(F_N \backslash T_\cala)$, the group $\IA_v(3)$ is a finite index subgroup of the group $\Out(G_v,\mathrm{Inc}_v)$. Moreover, $\phi$ is contained in $K$.
\end{thm}

We will also use a specific construction of Dehn twists which follows the work of Levitt~\cite{levitt2005}. 

Let $\calt$ be an $F_N$-equivariant isometry class of an $F_N$-tree $T$. The stabiliser $\Stab(\calt)$ of $\calt$ in $\IA_N(3)$ has a natural homomorphism $\Stab(\calt) \to \prod_{v\in V(F_N \backslash T)} \Out(G_v)$. By~\cite[Propositions~2.2, 2.3]{levitt2005} (see also~\cite[Proposition~2.7]{guirardel2021measure}), if every edge stabiliser is finitely generated, the kernel of this homomorphism consists of Dehn twists. Following the terminology of Levitt~\cite{levitt2005}, every bitwist is a Dehn twist. Note that, by Theorem~\ref{Thm:sesDehntwist}, for every Dehn twist $\phi \in \IA_N(3)$, the kernel of the natural homomorphism $C(\phi) \to \prod_{v\in V(F_N \backslash T_\cala)} \IA_v(3)$ consists of Dehn twists. We have in fact the following result.

\begin{lemma}\cite[Proposition~3.1]{levitt2005}\label{Lem:kernelabelian}
    Let $\calt$ be the $F_N$-equivariant isometry class of an $F_N$-tree $T$ with nontrivial finitely generated edge stabilisers. The kernel of $\Stab(\calt) \to \prod_{v\in V(F_N \backslash T)} \Out(G_v)$ is abelian and consists of Dehn twists.
\end{lemma}

\section{Property \texorpdfstring{$(C)$}{(C)} for \texorpdfstring{$\Out(F_N)$}{Out(FN)} and consequences}
\label{sec:PropC}
Let $N \geq 2$. Following Lück~\cite[Condition~3.1]{luck2009classifying}, we say that a group $G$ has \emph{Property $(C)$} if, for every infinite order element $h \in G$ and all $g \in G$ and $k,\ell \in \ZZ$, we have 
\[gh^kg^{-1}=h^{\ell} \Rightarrow |k|=|\ell|.\]

In this section, we prove the following.

\begin{prop}\label{Prop:property(C)}
    Let $N \geq 2$. The group $\Out(F_N)$ satisfies Property~$(C)$.
\end{prop}

\begin{proof}
    Let $\phi \in \Out(F_N)$ be an element of infinite order, and let $\psi \in \Out(F_N)$ and $k,\ell \in \ZZ$, be such that $\psi\phi^k\psi^{-1}=\phi^{\ell}$. We prove that $|k|=|\ell|$.

    Suppose first that $\phi \in \IA_N(3)$. Since $\IA_N(3)$ is a normal subgroup of $\Out(F_N)$, we also have $\psi\phi\psi^{-1} \in \IA_N(3)$. Since $\psi\phi^k\psi^{-1}=\phi^{\ell}$, a power of $\psi\phi\psi^{-1}$ commutes with a power of $\phi$. By  \Cref{Prop:powercommute}, the group $\langle \psi\phi\psi^{-1},\phi \rangle$ is abelian. Since $\IA_N(3)$ is torsion free, the group $\langle \psi\phi\psi^{-1},\phi \rangle$ is cyclic. Thus, there exists $m \in \ZZ$ such that either $\psi\phi\psi^{-1}=\phi^m$ or $\phi=\psi\phi^m\psi^{-1}$.

    We treat the case $\psi\phi\psi^{-1}=\phi^m$, the other one being similar. The group $\langle \psi,\phi \rangle$ is then a quotient of a metabelian Baumslag-Solitar group $\mathrm{BS}(1,m)$. In particular, it does not contain a nonabelian free group. By the Tits alternative for $\Out(F_N)$~\cite{BestvinaFeighnHandel00}, the group $\langle \psi,\phi \rangle$ is virtually abelian. 

    Let $n \geq 1$ be such that $\psi^n \in \IA_N(3)$. By  \Cref{Theo virtually abelian is abelian}, the group $\langle \psi^n, \phi \rangle$ is abelian. Recall that $\psi\phi\psi^{-1}=\phi^m$.   Thus, we have 
    \[\phi=\psi^n\phi\psi^{-n}=\phi^{m^n},\] 
    and so $m^n=1$ and $|m|=1$. As $\psi\phi\psi^{-1}=\phi^m$, this also implies that $|k|=|\ell|$.

    Suppose now that $\phi \notin \IA_N(3)$ and let $m \geq 1$ be such that $\phi^m \in \IA_N(3)$. Then we also have 
    \[\psi\phi^{mk}\psi^{-1}=\phi^{m\ell}.\] 
    By the previous case, we have $|k|=|\ell|$ and this concludes the proof.
\end{proof}

We now outline some consequences of  \Cref{Prop:property(C)}. Similar statements in the case of the mapping class group were proved by Juan-Pineda and Trujillo-Negrete~\cite{JuanPiTru2016}.

\begin{lemma}\label{Lemma:takingpowercenraliser}
    Let $\phi\in \IA_N(3)$. For every $n \geq 1$, we have $C_{\Out(F_N)}(\phi)=C_{\Out(F_N)}(\phi^n)$ and $N_{\Out(F_N)}(\langle \phi \rangle)=N_{\Out(F_N)}(\langle \phi^n \rangle)$.
\end{lemma}

\begin{proof}
    We prove the result for the centraliser, the proof for the normaliser being similar. Let $n\geq 1$. Since $C_{\Out(F_N)}(\phi) \subseteq C_{\Out(F_N)}(\phi^n)$, it suffices to prove the converse inequality. Let $\psi \in C_{\Out(F_N)}(\phi^n)$. Then $(\psi\phi\psi^{-1})^n=\phi^n$. Since $\phi,\psi\phi\psi^{-1}\in \IA_N(3)$, by  \Cref{Theo:roots}, we have $\psi\phi\psi^{-1}=\phi$.
\end{proof}

Note that \Cref{Lemma:takingpowercenraliser} is not true if we replace $\phi\in \IA_N(3)$ by $\phi \in \Out(F_N)$. Indeed, as $\Out(F_N)$ is centerless, any finite order element $\phi \in \Out(F_N)$ cannot satisfy $C_{\Out(F_N)}(\phi)=C_{\Out(F_N)}(\phi^n)$ for every $n\in \NN$ (see also~\cite{AndrewMartino2022} for an example with infinite order).

\begin{lemma}\cite[Lemma~4.2]{luck2009classifying}\label{Lemma:exaustioncommensurator}
    Let $G$ be a group satisfying Property~$(C)$ and let $C$ be an infinite virtually cyclic subgroup of $G$. For every $k \in \NN$, let $k!C$ be the subgroup of $C$ given by $\{h^{k!}\;|\; h \in C\}$. There exists a nested sequence \[N_G(C) \subseteq N_G(2!C) \subseteq \ldots \subseteq N_G(k!C) \subseteq \ldots\] such that 
    \[N_G[C]=\bigcup_{k \geq 1} N_G(k!C).\]
\end{lemma}

Recall that, for a group $G$ and a subgroup $C \subseteq G$, the group $N_G[C]$ is the commensurator of $C$ in $G$.

\begin{prop}\label{Prop:Commequalsnormaliser}
    Let $N \geq 2$, let $g\in \Out(F_N)$ be an infinite order element and let $n \geq 1$ be such that $g^n \in \IA_N(3)$. Then 
    \[N_{\Out(F_N)}[\langle g \rangle]=N_{\Out(F_N)}(\langle g^n \rangle).\]
\end{prop}

\begin{proof}
    See~\cite[Proposition~4.8]{JuanPiTru2016} for the mapping class group case. By  \Cref{Lemma:exaustioncommensurator}, we have 
    \[N_{\Out(F_N)}(\langle g\rangle) \subseteq N_{\Out(F_N)}(\langle g^{2!}\rangle) \subseteq \ldots \subseteq N_{\Out(F_N)}(\langle g^{k!}\rangle) \subseteq \ldots\] 
    and 
    \[N_{\Out(F_N)}[\langle g \rangle]=\bigcup_{k \geq 1} N_{\Out(F_N)}(\langle g^{k!}\rangle).\]
    Since $g^n \in \IA_N(3)$, by \Cref{Lemma:takingpowercenraliser}, for any $k \geq 1$, we have 
    \[N_{\Out(F_N)}(\langle g\rangle) \subseteq \ldots \subseteq N_{\Out(F_N)}(\langle g^{n!}\rangle)=N_{\Out(F_N)}(\langle g^{(n+k)!}\rangle).\] 
    By \Cref{Lemma:takingpowercenraliser} again, we have $N_{\Out(F_N)}(\langle g^{n!}\rangle)=N_{\Out(F_N)}(\langle g^n \rangle)$. Therefore, we see that 
    \[N_{\Out(F_N)}[\langle g \rangle]=N_{\Out(F_N)}(\langle g^n \rangle),\] 
    which concludes the proof.
\end{proof}

Combining  \Cref{Prop:Commequalsnormaliser} and~\cite[Corollary~4.2]{Guerch2022Roots}, we obtain the following.

\begin{thm}\label{Coro comm centraliser}
Let $\phi \in \IA_N(3)$. The commensurator of the cyclic group $\langle\phi\rangle$ in $\IA_N(3)$ is equal to its centraliser.
\hfill\qedsymbol
\end{thm}

\section{\texorpdfstring{Centralisers of elements in $\mathrm{Out}(F_N)$}{Centralisers of elements in Out(FN)}}\label{sec:centralisers}

Let $N \geq 2$. In this section, we study the centraliser of elements in $\IA_N(3)$. Centralisers of elements of $\IA_N(3)$ play a key role in the construction of a model for $\underline{\underline{E}}\IA_N(3)$ by the L\"uck--Weiermann push-out construction (see~\Cref{LuckWeiermann} below). This is why we need a precise description of the centraliser of an arbitrary element of $\IA_N(3)$. 

Let $\phi \in \IA_N(3)$ be of infinite order and let $\cala=\{[\Fix(\Phi)]\}_{\Phi \in \phi}$. Recall that $\cala$ is a finite set of conjugacy classes of finitely generated subgroups of $F_N$ and that $\cala$ is stabilised by the centraliser $C(\phi)$ of $\phi$ in $\IA_N(3)$. 

The study of the centraliser of $F_N$ will be divided into two parts, depending on whether $F_N$ is one-ended relative to $\cala$ or not.

\subsection{The one-ended case}\label{sec:oneended}

In this section, suppose that $F_N$ is one-ended relative to $\cala=\{[\Fix(\Phi)]\}_{\Phi \in \phi}$. We will study the action of $C(\phi)$, the centraliser of $\phi$ in $\IA_N(3)$ on the JSJ tree $T_\cala$ associated with $\cala$. The main result is the following.

\begin{thm}\label{Thm:sesoneendedcase}
    Let $\phi \in \IA_N(3)$. Suppose that $F_N$ is one-ended relative to $\cala$. Suppose also that $\phi$ is not a Dehn twist. Let $C(\phi)$ be the centraliser of $\phi$ in $\IA_N(3)$. Recall the partition $VT_\cala=V_1\coprod V_2$.

    \begin{enumerate}
        \item The group $C(\phi)$ fits into an exact sequence 
        \[1 \to K' \to C(\phi) \to \ZZ \times \prod_{v \in V_2} \Out(G_v)\] where $K'$ is abelian. 

    \item The image of the projection on the first coordinate $$C(\phi) \to \ZZ$$ is generated by a root of $\phi$. The kernel $K$ satisfies one of the followings.

    \begin{enumerate}
        \item The group $K$ is isomorphic to a subgroup of \[\Out(A) \times \Out(B),\] where $A,B \subseteq F_N$ are such that $\rank(A)+\rank(B)=N+1$ and $\rank(A),\rank(B) \leq N-1$.
        \item The group $K$ is isomorphic to a subgroup of \[\Out(A \ast \langle tst^{-1}\rangle,[s],[tst^{-1}]),\] where $s,t \in F_N$ $A \subseteq F_N$, $\rank(A)=N-1$, $s \in A$ and $t$ is a basis element of $F_N$.
    \end{enumerate}
    \end{enumerate}
\end{thm}

\begin{proof}
    Consider the action of $C(\phi)$ on $T_\cala$. This action gives an exact sequence \[1\to K'\to C(\phi )\to \prod_{v\in VF_N\backslash T_\cala} \Out(G_v).\] By Lemma~\ref{Lem:kernelabelian}, since every edge of $T_\cala$ has infinite cyclic edge stabiliser, the kernel $K'$ is abelian.
    
    Recall the partition $VT_\cala=V_1\coprod V_2$ given by Theorem~\ref{Thm:JSJ}. Note that $\phi\in \Out(F_N,\cala^{(t)})$.

    Suppose towards a contradiction that $VT_\cala=V_2$. By Theorem~\ref{Thm:JSJ}~$(3)(d)$, since $\phi\in \Out(F_N,\cala^{(t)})$, for every $v \in V_2$, the image of $\phi$ in $\Out(G_v)$ is trivial. Hence $\phi$ is contained in the kernel of the homomorphism $C(\phi) \to \prod_{v\in VT} \Out(G_v)$. But the kernel of this homomorphism consists of Dehn twists by \Cref{Lem:kernelabelian}. This contradicts the assumption made on $\phi$. Thus, the set $V_1$ is nontrivial. 
    
    Let $v\in V_1$ and let $\Sigma_v$ be the associated compact hyperbolic surface given by Theorem~\ref{Thm:JSJ}~$(3)(a)$. By Theorem~\ref{Thm:JSJ}~$(3)(b)$, for every $[A]\in \cala$, the intersection of $A$ with $G_v$ is contained in a boundary component of $\phi$. Recall that every $\phi$-periodic conjugacy class of $F_N$ is in fact fixed by Theorem~\ref{Theo:freefactorconjclassfixed}~$(2)$. Thus, by definition of $\cala$, the mapping class of $\Sigma_v$ induced by $\phi$ does not virtually preserve the homotopy class of any curve nonhomotopic to a boundary component. Therefore, the image of $\phi$ in $\MCG(\Sigma_v)$ is a pseudo-Anosov homeomorphism. This implies that its centraliser in $\MCG(\Sigma_v)$ is virtually cyclic. 

    Note that, since $C(\phi) \subseteq \IA_N(3)$, by Theorem~\ref{Theo:freefactorconjclassfixed}~$(2)$, the image of $C(\phi)$ in $\Out(G_v)$ is torsion free. Combining this remark and the above paragraph, we see that the image of $C(\phi)$ in $\Out(G_v)$ is infinite cyclic, generated by a root of $\phi$. Hence, for every $v \in V_1$, the image $C(\phi) \to \Out(G_v)$ is infinite cyclic, generated by a root of $\phi$. Thus, we have the following exact sequence \[1 \to K' \to C(\phi) \to \ZZ \times \prod_{w \in V_2} \Out(G_w)\] where $K'$ is abelian. This proves Assertion~$(1)$.

    We now prove Assertion~$(2)$. Let $v\in V_1$ and let $K=\ker(C(\phi) \to \Out(G_v))$. 

    Note that $\Sigma_v$ is not homeomorphic to a pair of pants as $\MCG(\Sigma_v)$ is infinite (see~\Cref{Thm:JSJ}~$(3)(a)$). Thus, there exists a closed geodesic curve $\gamma$ in $\Sigma_v$ which is not homotopic to a boundary component. This curve induces a splitting $S_v$ of $\Sigma_v$. Since edge groups are all contained in boundary components of $\Sigma_v$, one can blow up $S_v$ at $v$ to obtain a splitting $T'$ of $F_N$ such that $T$ is obtained from $T'$ by collapsing the orbit of an edge $e$. Note that the stabiliser of $e$ is infinite cyclic and its conjugacy class corresponds to the conjugacy class associated with $\gamma$. Moreover, $T'$ is preserved by $K$ since $K$ acts as the identity on $\pi_1(\Sigma_v)$.
    
    Let $U$ be the splitting obtained from $T'$ by collapsing every orbit of edges except the one of $e$. Then, $U$ is preserved by $K$. Thus, we have a homomorphism $K\to \prod_{w\in F_N \backslash U} \Out(G_w)$ whose kernel is infinite cyclic and generated by a twist $D$ about $e$. Since $G_e$ is contained in the conjugacy class of $\gamma$, the twist $D$ does not have a representative which acts as the identity on $\pi_1(\Sigma_v)$.  As $K$ acts as the identity on $\pi_1(\Sigma_v)$, the homomorphism $K\to \prod_{v\in F_N \backslash U} \Out(G_w)$ is injective. 
    
    Since $U$ has one orbit of edges, $U$ induces one of the following splittings of $F_N$: either $F_N=A\ast_\ZZ B$ where $\rank(A)+\rank(B)=N+1$ and $\rank(A),\rank(B) \leq N-1$, which yields case (2)(a), or $F_N=(A \ast \langle x \rangle) \ast_\ZZ$ where $\rank(A)=N-1$ and $x \in F_N$. Moreover, in the second case, there exist $s \in A$ and a basis element $t$ of $F_N$ such that $x=tst^{-1}$ and the vertex stabiliser of $U$ is conjugate to $A \ast \langle tst^{-1} \rangle$. Since $K$ preserves $U$, the image of $K$ in $\Out(A \ast \langle tst^{-1} \rangle)$ preserves $s$ and $tst^{-1}$ which yields case (2)(b). This concludes the proof.
\end{proof}

\subsection{The infinitely-ended-case}\label{sec:infinitelyended}

Let $\phi \in \IA_N(3)$. Suppose now that $F_N$ is not one-ended relative to $\cala=\{[\Fix(\Phi)]\}_{\Phi \in \phi}$. Let $\calf_\cala$ be the minimal free factor system of $F_N$ such that for every $[A]\in \cala$, there exists $[B]\in \calf_\cala$ such that $A\subseteq B$. By minimality of $\calf_\cala$, we have $C(\phi) \subseteq \Out(F_N,\calf_\cala)$. We will consider $C(\phi)$-invariant free factor systems $\calf$ of $F_N$ such that $\calf_\cala \leq \calf$ in order to obtain a description of $C(\phi)$.

\begin{thm}\label{thm:shortexactsequencecentraliser}
    Let $\phi\in \IA_N(3)$. Let $C(\phi)$ be the centraliser of $\phi$ in $\IA_N(3)$. Let $\calf_\cala \leq \calf$ be a maximal proper $C(\phi)$-invariant free factor system.
    \begin{enumerate}
        \item If $\calf$ is nonsporadic, there exists a surjective homomorphism $$C(\phi) \to \ZZ$$ whose image is generated by a root of $\phi$ and such that the kernel $K$ of this homomorphism is isomorphic to a subgroup of $\Aut(A) \times \Aut(B)$, where $A,B \subseteq F_N$ are nontrivial subgroups such that $F_N=A \ast B$.
        \item If $\calf$ is sporadic, then $\calf=\{[A_1],[A_2]\}$ and there exists a homomorphism 
        \[C(\phi) \to \IA(A_1,3) \times \IA(A_2,3)\] 
        whose kernel is a finite index subgroup of a direct product of two finitely generated free (maybe trivial or cyclic) groups.
    \end{enumerate}
\end{thm}

\begin{remark}
    A key point in the proof of \Cref{thm:shortexactsequencecentraliser} is the fact that every free factor system which has a finite orbit under iteration of an element of $\IA_N(3)$ is in fact fixed (see~\Cref{Theo:freefactorconjclassfixed}). Therefore, it is not clear how to deduce an exact sequence similar to the one of \Cref{thm:shortexactsequencecentraliser} for centralisers in $\Out(F_N)$ instead of in $\IA_N(3)$.
\end{remark}

\begin{remark}\label{Rmk:rankfreefactorsystem}
    In  \Cref{thm:shortexactsequencecentraliser}, the number of elements in the free factor system $\calf$ depends on whether Case~$1$ or Case~$2$ holds. In Case~$2$ the group $A_1$ is never trivial but $A_2$ might be trivial.
    We also have a control on the ranks of the free factors appearing in $\calf$. Indeed, in Case~$2$, for every $[A]\in \calf$, the rank of $A$ is bounded by $N-1$. Moreover,  
    \[\sum_{[A]\in \calf}\mathrm{rank}(A) \leq N.\] 
\end{remark}

In order to prove \Cref{thm:shortexactsequencecentraliser}, we need some results regarding the stabiliser of a relative arational tree in the Gromov boundary of a relative free factor graph. The first one is a combination of a result extracted from \cite{HorbezWade20}, where it is attributed to Guirardel and Levitt and a result of Guirardel and Horbez~\cite{Guirardelhorbez19}.

\begin{prop}\cite{Guirardelhorbez19,HorbezWade20}\label{Prop existence splitting arational tree}
Let $H$ be a subgroup of $\IA_N(3)$. Let $\calf$ be a maximal $H$-invariant free factor system. Suppose that $\calf$ is nonsporadic and that $H$ has a finite index subgroup which fixes the homothety class of an $(F_N,\mathcal{F})$-arational tree $T$. 

\begin{enumerate}
    \item We have a homomorphism $H \to \ZZ \times \prod_{[A]\in \calf} \IA(A,3)$ whose kernel is abelian and consists of Dehn twists.
    \item The image of the projection $p \colon H \to \ZZ$ is surjective and generated by a root of any fully irreducible outer automorphism relative to $\calf$ contained in $H$. 
    \item The kernel $K$ of $p$ is isomorphic to a subgroup of $\Aut(A) \times \Aut(B)$, where $A,B \subseteq F_N$ are nontrivial subgroups such that $F_N=A \ast B$. 
    \item There exists a nonperipheral subgroup $C \subseteq F_N$ such that, for every $k \in K$, the outer automorphism $k$ has a representative fixing $C$ elementwise.
\end{enumerate}
\end{prop}

\begin{proof}
    By maximality of $\calf$ and \Cref{Theo fully irreducible contained}, the group $H$ contains a fully irreducible outer automorphism relative to $\calf$. Let $H_0$ be a finite index subgroup of $H$ which fixes the homothety class of $T$. By  \Cref{Lem stratching factor cyclic}, the group $H_0$ splits as a semi-direct product $H' \rtimes \ZZ$, where $H'$ is contained in the isometric stabiliser of $T$. Moreover, the $\ZZ$ factor is generated by a root of any fully irreducible outer automorphism relative to $\calf$ contained in $H_0$.
    
    By~\cite[Lemmas~5.3,~5.6, Theorem~5.4]{HorbezWade20}, the commensurator of $H'$ in $\Out(F_N,\calf) \cap \IA_N(3)$, denoted by $N_{\IA_N(\calf,3)}[H']$, preserves the $F_N$-equivariant isometry class of an $F_N$-tree $S$. 
    
    The tree $S$ satisfies the following properties. Edge stabilisers in $F_N$ are infinite. The quotient graph $F_N \backslash S$ is a tree with one central vertex, $v_0$, adjacent to every other vertex. The stabiliser of $v_0$ in $F_N$ is finitely generated. Moreover, if $v\in V(F_N \backslash S)-\{v_0\}$, then the conjugacy class of the stabiliser $G_v$ of $v$ in $F_N$ is contained in $\calf$.
    
    By \Cref{Theo:freefactorconjclassfixed}, the group $N_{\IA_N(\calf,3)}[H']$ preserves the conjugacy class of every $G_v$ with $v\in V(F_N \backslash S)-\{v_0\}$. Thus, the group $N_{\IA_N(\calf,3)}[H']$ acts trivially on the quotient graph $F_N \backslash S$. Therefore we have a homomorphism 
    \[N_{\IA_N(\calf,3)}[H'] \to  \Out(G_v) \times \prod_{[A] \in \mathcal{F}} \Out(A)\] 
    induced by the action on the vertex stabilisers. Since $N_{\IA_N(\calf,3)}[H'] \subseteq \IA_N(3)$ and since we are considering the restriction homomorphism on vertex stabilisers, the restriction of the image of $H$ in $\prod_{[A] \in \mathcal{F}} \Out(A)$ is contained in $\prod_{[A] \in \mathcal{F}} \IA(A,3)$.

    By~\cite[Theorem~5.4]{HorbezWade20}, the image of $H'$ in $\Out(G_v)$ is trivial and the image of $H_0=H' \rtimes \ZZ$ is infinite cyclic.
    
    Thus, the group $N_{\IA_N(\calf,3)}[H']$ fits into an exact sequence \[1\to K' \to N_{\IA_N(\calf,3)}[H'] \to \Out(G_v) \times \prod_{[A] \in \mathcal{F}} \IA(A,3),\] where $G_v$ is a nonabelian free subgroup of $F_N$ such that the image of $H_0$ in $\Out(G_v)$ is isomorphic to $\ZZ$. 

    Since $H'$ is a normal subgroup of a finite index subgroup of $H$, the group $H$ is contained in $N_{\IA_N(\calf,3)}[H']$. Therefore, we have an exact sequence \[1\to K_0 \to H \to \Out(G_v) \times \prod_{[A] \in \mathcal{F}} \IA(A,3).\] The kernel $K_0$ is the subgroup of $H$ acting trivially on the vertex groups of $T$. Since $T$ has nontrivial edge stabilisers, by \Cref{Lem:kernelabelian}, the group $K_0$ consists of Dehn twists and is abelian.
    
     Thus, in order to prove Assertion~$(1)$, it suffices to prove that the image of $H$ in $\Out(G_v)$ is infinite cyclic. Since the image of $H_0$ in $\Out(G_v)$ is infinite cyclic and since $H_0$ is a finite index subgroup of $H$, we see that the image of $H$ in $\Out(G_v)$ is virtually infinite cyclic. 
    
    Thus, it suffices to prove that the image of $H$ in $\Out(G_v)$ is torsion free. Let $\psi \in H$ whose image in $\Out(G_v)$ is finite. Thus, $\psi$ has a power which preserves the conjugacy class of every element of $G_v$. By  \Cref{Theo:freefactorconjclassfixed}, the outer automorphism $\psi$ preserves the conjugacy class of every element of $G_v$. In particular, the image of $\psi$ in $\Out(G_v)$ is trivial. This implies that the image of $H$ in $\Out(G_v)$ is virtually infinite cyclic and torsion free, hence is infinite cyclic. This proves Assertion~$(1)$. 
    
    As explained above, the image of $H$ in $\Out(G_v)$ is generated by a root of any fully irreducible outer automorphism relative to $\calf$ whose power is contained in $H_0$. Thus, it is generated by a root of any fully irreducible outer automorphism relative to $\calf$ contained in $H$. This proves Assertion~$(2)$.

    We now prove that the kernel $K$ of $p \colon H \to \Out(G_v)$ satisfies Assertions~$(3),(4)$. Note that every element of $K$ has a representative which fixes elementwise $G_v$, which is a nonperipheral subgroup. Assertion~$(4)$ follows.

    It remains to prove Assertion~$(3)$. It suffices to prove that $K$ is isomorphic to a subgroup of $\Aut(A) \times \Aut(B)$, where $A,B \subseteq F_N$ are nontrivial subgroups such that $F_N=A \ast B$. By~\cite[Lemma~5.6]{HorbezWade20}, the edges groups of $S$ induce a nonsporadic free factor system of $G_v$. Thus, there exist nontrivial subgroups $C,D\subseteq G_v$ such that $G_v=C \ast D$ and, for every $e \in ES$, a conjugate of the group $G_e$ is contained in either $C$ or $D$. Since $K$ acts trivially on $G_v$, it also preserves this decomposition. 
    
    Let $S'$ be the splitting obtained from $S$ by blowing up at $v$ the splitting $S_v$ induced by $G_v=C \ast D$ and attaching the edges groups accordingly. Then $S'$ is preserved by $K$ since $K$ preserves both $S$ and $S_v$. Moreover, $S'$ has a unique orbit of an edge $e$ with trivial stabiliser. Since $v$ meets every orbit of edges in $S$, the image of $e$ in $\overline{F_N \backslash S'}$ is a separating edge. 

    Let $U$ be the splitting obtained from $S'$ by collapsing every orbit of edges except the one of $e$. Then $U$ is preserved by $K$ since $K \subseteq \IA_N(3)$. Moreover, the decomposition of $F_N$ induced by $U$ is $F_N=A \ast B$, where $A,B \subseteq F_N$ are nontrivial subgroups such that $F_N=A \ast B$. The stabiliser of this splitting is isomorphic to $\Aut(A) \times \Aut(B)$ by a result of Levitt~\cite{levitt2005}. This concludes the proof.
\end{proof}

\begin{lemma}\label{Lem:centraliserfixespointboundary}
Let $N \geq 2$ and let $\phi \in \IA_N(3)$. Let $\calf_\cala \leq \mathcal{F}$ be a  maximal $C(\phi)$-invariant free factor system. Suppose that $\mathcal{F}$ is nonsporadic. The element $\phi$ is fully irreducible relative to $\calf$. Moreover, the group $C(\phi)$ virtually fixes a point in $\partial_{\infty} \FF(F_N,\mathcal{F})$.
\end{lemma}

\begin{proof}
By  \Cref{Theo fully irreducible contained}, the group $C(\phi)$ contains a fully irreducible element $\psi$ relative to $\mathcal{F}$. 

We claim that $\phi$ is also a fully irreducible element relative to $\mathcal{F}$. Indeed, by \Cref{Theo loxo free factor}, the element $\psi$ is a loxodromic element of $\FF(F_N,\mathcal{F})$. Thus, $\psi$ fixes exactly two points $T_+,T_-$ in $\partial_{\infty} \FF(F_N,\mathcal{F})$. Since $\phi$ commutes with $\psi$, the element $\phi$ preserves $\{T_+,T_-\}$. 

By~\Cref{prop fix arational boundary}, the element $\phi$ virtually fixes the homothety class of an arational $(F_N,\calf)$-tree. Thus, we can apply \Cref{Prop existence splitting arational tree} to see that $\langle \phi \rangle$ fits in an exact sequence $$1 \to K \to \langle \phi \rangle \to \ZZ $$ whose kernel $K$ fixes elementwise a nonperipheral group. 

Since $\calf_\cala \leq \calf$, we see that $\phi$ does not fix elementwise a nonperipheral subgroup. In particular, the group $K$ is trivial. Since the image of $\langle \phi \rangle \to \ZZ$ is generated by any fully irreducible outer automorphism relative to $\calf$ by~\Cref{Prop existence splitting arational tree}~$(2)$, we see that $\phi$ itself is fully irreducible relative to $\calf$.

The moreover part follows from the fact that $C(\phi)$ must preserve the attracting and repelling fixed points of $\phi$ in $\partial_{\infty} \FF(F_N,\mathcal{F})$.
\end{proof}

\begin{proof}[Proof of \Cref{thm:shortexactsequencecentraliser}] Let $\calf \geq \calf_\cala$ be a maximal $C(\phi)$-invariant free factor system. 

\medskip

\noindent{\bf Case 1. } \emph{Suppose that $\calf$ is sporadic.}

Thus, we have $\calf=\{[A],[B]\}$ where $A$ and $B$ might be equal. By for instance \cite[Proposition~4.2]{levitt2005}, the stabiliser of $\calf$ in $\IA_N(3)$ is isomorphic to a finite index subgroup of either $\Aut(A) \times \Aut(B)$ if $[A]\neq [B]$ or of $A \rtimes \Aut(A)$ otherwise. In both cases, we have a homomorphism $C(\phi) \to \prod_{[C]\in \calf} \Out(C)$ whose kernel is a finite index subgroup of a direct product of two free (maybe cyclic or trivial) groups. Since $C(\phi) \subseteq \IA_N(3)$, the image of $C(\phi) \to \prod_{[C]\in \calf} \Out(C)$ is contained in $\prod_{[C]\in \calf} \IA(C,3)$. Thus, it remains to show that both such free groups in the kernel are finitely generated. We treat both cases separately. 

Suppose that $\Stab(\calf)$ is isomorphic to a finite index subgroup of $\Aut(A) \times \Aut(B)$. Let $\Phi \in\phi$ be the unique automorphism in the outer class $\phi$ such that $\Phi(A)=A$ and $\Phi(B)=B$. Then the kernel of the homomorphism $C(\phi) \to\Out(A) \times \Out(B)$ is isomorphic to the intersection of $\IA_N(3)$ with a subgroup isomorphic to $\Fix(\Phi|_A) \times \Fix(\Phi|_B)$. In particular, both direct factors are finitely generated by~\cite{BesHan92}. 

Suppose now that $\Stab(\calf)$ is isomorphic to $A \rtimes \Aut(A)$. In that case, we have $F_N=A \ast \langle g \rangle$ for some $g\in F_N$. Let $\Phi \in \phi$ be the unique representative of $\phi$ sending $A$ to $A$ and $g$ to $ga$ with $a \in A$. Then the kernel of the homomorphism $C(\phi) \to\Out(A)$ is isomorphic to the intersection of $\IA_N(3)$ with a subgroup isomorphic to $\Fix(\Phi|_A) \times \Fix(\mathrm{ad}_{a^{-1}} \circ \Phi|_A)$, so that both direct factors are finitely generated. This concludes the proof when $\calf$ is sporadic.

\medskip

\noindent{\bf Case 2. } \emph{Suppose that $\calf$ is nonsporadic.}

By \Cref{Lem:centraliserfixespointboundary}, the set of fixed points of $C(\phi)$ in $\partial_{\infty}\FF(F_N,\calf)$ is nonempty. By~\Cref{prop fix arational boundary}, the group $C(\phi)$ virtually fixes the homothety class of an $(F_N,\calf)$-arational tree. Thus, we can apply \Cref{Prop existence splitting arational tree} in order to obtain the desired homomorphism $C(\phi)\to \ZZ$. Since $\phi$ is fully irreducible relative to $\calf$ by \Cref{Lem:centraliserfixespointboundary}, a root of $\phi$ generates the image of the homomorphism. This concludes the proof.
\end{proof}

Combining \Cref{Prop existence splitting arational tree}, Theorems~\ref{Thm:sesoneendedcase} and~\ref{thm:shortexactsequencecentraliser}, we obtain the following properties of centralisers of elements of $\IA_N(3)$.

\begin{thm}\label{thm:classificationcentralisers}
    Let $N \geq 2$ and let $\phi \in \IA_N(3)$. The centraliser $C(\phi)$ of $\phi$ in $\IA_N(3)$ satisfies one of the followings.
    \begin{enumerate}
        \item The outer automorphism $\phi$ is a Dehn twist. There exist a JSJ tree $T$ preserved by $C(\phi)$ and a short exact sequence 
    \[1 \to K \to C(\phi) \to \prod_{v\in V(F_N \backslash T)} \IA_v(3) \to 1, \] where $K$ is a free abelian group whose dimension is equal to $|E(F_N \backslash T)|$ and, for every $v\in V(F_N \backslash T)$, the group $\IA_v(3)$ is a finite index subgroup of the group $\Out(G_v,\mathrm{Inc}_v)$. Moreover, $\phi$ is contained in $K$.
        \item There exist $A_1,A_2 \subseteq F_N$ with $F_N=A_1 \ast A_2$, $\rank(A_1),\rank(A_2) \leq N-1$, and a homomorphism
        \[C(\phi) \to \IA(A_1,3) \times \IA(A_2,3)\] 
         whose kernel is a finite index subgroup of a direct product of two finitely generated free (maybe trivial or cyclic) groups.
        \item There exist $A_1,\ldots, A_k,B \subseteq F_N$ nontrivial with $F_N=A_1 \ast \ldots A_k \ast B$, and a homomorphism \[C(\phi) \to \ZZ \times \prod_{i=1}^k \IA(A_i,3)\] whose kernel is abelian and $\phi$ projects onto the $\ZZ$ factor.
        \item There exist a JSJ tree $T$ preserved by $C(\phi)$, a partition $VT=V_1\coprod V_2$, and a homomorphism \[C(\phi) \to \ZZ \times \prod_{v \in V_2} \Out(G_v)\] whose kernel is abelian and $\phi$ projects onto the $\ZZ$ factor.
        \end{enumerate}
\end{thm}
\begin{proof}
    To make this explicit suppose first that $\phi$ is a Dehn twist. Then we are in case $(1)$ and the short exact sequence follows from \Cref{Thm:sesDehntwist}. 
    
    Suppose now that $\phi$ is not a Dehn twist. Let $\cala=\{[\Fix(\Phi)]\}_{\Phi \in \phi}$. Suppose that $F_N$ is one-ended relative to $\cala$. Then we can apply \Cref{Thm:sesoneendedcase} to obtain case $(4)$. 
    
    Suppose that $F_N$ is not one-ended relative to $\cala$ and let $\calf \geq \calf_\cala$ be a maximal $C(\phi)$-invariant free factor system. If $\calf$ is sporadic, we can apply \Cref{thm:shortexactsequencecentraliser}~$(2)$ to get case $(2)$.
    
    Suppose that $\calf$ is nonsporadic. By Lemma~\ref{Lem:centraliserfixespointboundary}, the group $C(\phi)$ virtually fixes a point in $\partial_{\infty}\FF(F_N,\calf)$ and $\phi$ is fully irreducible relative to $\calf$. By~\Cref{prop fix arational boundary}, the group $C(\phi)$ virtually fixes the homothety class of an $(F_N,\calf)$-arational tree. Thus, we can apply \Cref{Prop existence splitting arational tree}~$(1)$ to get the homomorphism of case $(3)$. Note that a root of $\phi$ projects onto the $\ZZ$ factor by \Cref{Prop existence splitting arational tree}~$(2)$ since $\phi$ is fully irreducible relative to $\calf$.
\end{proof}

\subsection{Weyl groups}
We now adapt Theorems~\ref{Thm:sesoneendedcase} and~\ref{thm:shortexactsequencecentraliser} to the Weyl group $W(\phi)=C(\phi)/\langle \phi \rangle$ in $\IA_N(3)$ of an infinite order element $\phi \in \IA_N(3)$. Let $\calf$ be a free factor system as in \Cref{thm:shortexactsequencecentraliser}. For every $[A]\in \calf$, since $\IA_N(3)$ is torsion free, the image $\phi|_A$ of $\phi$ in $\IA(A,3)$ is either trivial or infinite. Let $\calf_\infty$ be the subset of $\calf$ consisting of all $[A]\in \calf$ such that $\phi|_A$ is infinite and let $\calf_T=\calf- \calf_\infty$. Let $H_T=\prod_{[A] \in \calf_T} \IA(A,3)$ and let $H_\infty=\prod_{[A] \in \calf_\infty} C(\phi|_A)$. We denote by $\rho_\infty \colon C(\phi) \to H_\infty$ and $\rho_T \colon C(\phi) \to H_T$ the homomorphisms given by  \Cref{thm:shortexactsequencecentraliser}~$(2)$. 

\begin{corollary}\label{Coro:sesWeylgroup}
    Let $N \geq 2$, let $\phi \in \IA_N(3)$ be a root-closed element of infinite order and let $\cala=\{[\Fix(\Phi)]\}_{\Phi \in \phi}$. Let $\calf \geq \calf_\cala$ be a (possibly trivial) maximal $C(\phi)$-invariant free factor system. The Weyl group $W(\phi)=C(\phi)/\langle \phi \rangle$ of $\phi$ in $\IA_N(3)$ satisfies one of the following.

    \begin{enumerate}

        \item The element $\phi$ is a Dehn twist.
        \item The group $F_N$ is one-ended relative to $\cala$. The group $W(\phi)$ is isomorphic to a subgroup of  $\Out(A) \times \Out(B)$, where $A,B \subseteq F_N$ are such that $\rank(A)+\rank(B)=N+1$ and $\rank(A),\rank(B) \leq N-1$.
        \item The group $F_N$ is one-ended relative to $\cala$. The group $W(\phi)$ is isomorphic to a subgroup $\Out(A \ast \langle tst^{-1}\rangle,[s],[tst^{-1}])$, where $A \subseteq F_N$, $\rank(A)=N-1$, $s \in A$ and $t$ is a basis element of $F_N$.
        \item The free factor system $\calf$ is nonsporadic. The group $W(\phi)$ is isomorphic to a subgroup of $\Aut(A) \times \Aut(B)$, where $A,B \subseteq F_N$ are nontrivial subgroups such that $F_N=A \ast B$.
        \item The free factor system $\calf$ is sporadic and $\rho_\infty(\phi)$ is infinite. Then $W(\phi)$ fits into an exact sequence 
        \[1 \to K \to W(\phi) \to H_\infty/\rho_\infty(\langle \phi \rangle) \times H_T,\]
        where $K$ is a finite index subgroup of a direct product of two finitely generated free groups.
        \item The free factor system $\calf$ is sporadic and $\rho_\infty(\phi)$ is trivial. Then $W(\phi)$ fits into an exact sequence 
        \[1 \to K \to W(\phi) \to \prod_{[A]\in \calf} \IA(A,3),\] 
        where $K$ is a direct product of a finitely generated free (maybe cyclic or trivial) group and a finite group.
    \end{enumerate}
\end{corollary}

\begin{proof}

    We will frequently use the fact that, if $C(\phi)$ maps onto $\ZZ$ with image generated by $\phi$, then $C(\phi)$ splits as $C(\phi)=K \times \langle \phi \rangle$, where $K$ is the kernel of this homomorphism. 
    
    Suppose that $\phi$ is not a Dehn twist. We begin with the case that $F_N$ is one-ended relative to $\cala$. By Theorem~\ref{Thm:sesoneendedcase}~(2), we have a homomorphism $C(\phi) \to \ZZ$ whose image is generated by a root of $\phi$. Since $\phi$ is root-closed, the image is generated by $\phi$. Thus, the Weyl group $W(\phi)$ is isomorphic to the kernel $K$ of this homomorphism. By Theorem~\ref{Thm:sesoneendedcase}~(2)(a) or (b), the group $W(\phi)$ satisfies either Case~$(2)$ or $(3)$ respectively.

    We now deal with the case that $F_N$ is not one-ended relative to $\cala$. Consider the maximal $C(\phi)$-invariant free factor system $\calf \geq \calf_\cala$. 
    
    Suppose first that $\calf$ is nonsporadic. \Cref{thm:shortexactsequencecentraliser}~$(1)$ gives a homomorphism whose image is generated by a root of $\phi$. As in the proof of Cases~$(2)$ and $(3)$, the group $W(\phi)$ is isomorphic to the kernel of this homomorphism, which leads Case~$(4)$.

    Suppose now that $\calf$ is sporadic. \Cref{thm:shortexactsequencecentraliser}~$(2)$ gives a homomorphism $C(\phi) \xrightarrow{(\rho_\infty,\rho_T)} H_\infty \times H_T$. It induces a quotient map 
    \[C(\phi) \to (H_\infty \times H_T)/_{(\rho_\infty,\rho_T)(\langle \phi \rangle)}.\] 
    Since $\rho_T(\phi)$ is trivial, this gives a map 
    \[C(\phi) \to H_\infty/\rho_\infty(\langle \phi \rangle) \times H_T\] 
    which induces a quotient map 
    \[W(\phi) \xrightarrow{\rho_W} H_\infty/\rho_\infty(\langle \phi \rangle) \times H_T.\]
    
    When $\rho_\infty (\phi)$ is infinite, the kernel of $\rho_W$ is exactly the same as the homomorphism $(\rho_\infty,\rho_T)$, so Case~$(5)$ follows. 
    
   Suppose that $\rho_\infty (\phi)$ is trivial. Then $\phi$ is contained in the kernel of $(\rho_\infty,\rho_T)$ and the kernel of $\rho_W$ is $\ker((\rho_\infty,\rho_T))/\langle \phi \rangle$. Moreover, $H_\infty$ is the trivial group.
   
    By  \Cref{thm:shortexactsequencecentraliser}~$(2)$, the kernel $\ker((\rho_\infty,\rho_T))$ is isomorphic to a direct product of two finitely generated free groups. Note that that the centraliser of an infinite element in a direct product of two finitely generated free groups is contained in a group isomorphic to $F\times \ZZ$, where $F$ is a finitely generated free group. Thus, the quotient $\ker((\rho_\infty,\rho_T))/\langle \phi \rangle$ is isomorphic to a direct product of a finitely generated free group and a finite group. This shows Case~$(6)$.
\end{proof}

\section{\texorpdfstring{Centralisers of elements in $\mathrm{Out}(F_3)$}{Centralisers of elements in Out(F3)}} \label{sec:OutF3}
 
In this section, we focus on the study of outer automorphisms of $F_3$. We prove the following. 

\begin{thm}\label{Theo type VF}
Let $\phi \in \IA_3(3)$. The centraliser $C(\phi)$ of $\phi$ in $\IA_3(3)$ is of type $\mathsf{VF}$. Moreover, one of the following holds.
\begin{enumerate}
\item The centraliser of $\phi$ is abelian.
\item The centraliser of $\phi$ is isomorphic to $F \times \ZZ$ where $F$ is a finitely generated free group.
\item The centraliser of $\phi$ is isomorphic to a direct product $H\times \ZZ$ where $H$ is a finite index subgroup of a direct product of two finitely generated free groups.
\item The outer automorphism $\phi$ is a Dehn twist. There exist a JSJ tree $T$ preserved by $C(\phi)$ and a short exact sequence 
    \[1 \to K \to C(\phi) \to \prod_{v\in V(F_3 \backslash T)} \IA_v(3) \to 1, \] 
    where $K$ is a free abelian group whose dimension is equal to $|E(F_3 \backslash T)|$ and, for every $v\in V(F_3 \backslash T)$, the group $\IA_v(3)$ is a finite index subgroup of the group $\Out(G_v,\mathrm{Inc}_v)$. Moreover, $\phi$ is contained in $K$.
\end{enumerate}
\end{thm}

We highlight the following immediate corollary.

\begin{corollary}\label{cor type VF}
    Let $\phi\in\IA_3(3)$.  The centraliser $C_{\Out(F_3)}(\phi)$ is of type $\mathsf{VF}$.
\end{corollary}

The proof of  \Cref{Theo type VF} is decomposed into several propositions. The idea is to consider a maximal $C(\phi)$-invariant free factor system $\mathcal{F}$ and to treat separately the cases when $\mathcal{F}$ is sporadic or not. Observe that, since we are considering a nonabelian free group of rank $3$, the free factor system $\mathcal{F}$ is sporadic if and only if it contains the conjugacy class of a nonabelian free factor.

\begin{lemma}\label{Lem:typevfnonsporadic}
Let $\phi \in \IA_3(3)$. Let $\mathcal{F}$ be a  maximal $C(\phi)$-invariant free factor system. Suppose that $\mathcal{F}$ is nonsporadic. If $\phi$ is not a Dehn twist then $C(\phi)$ is abelian.
\end{lemma}

\begin{proof}
We claim that $\phi$ is fully irreducible relative to $\calf$. Indeed, by maximality of $\calf$ and \Cref{Theo loxo free factor}, the group $C(\phi)$ contains a loxodromic element $\psi$ of $\FF(F_N,\calf)$. As $\psi$ commutes with $\phi$, the element $\phi$ must fix the attracting point of $\psi$ in $\partial_{\infty}\FF(F_N,\calf)$. By \Cref{prop fix arational boundary}, $\phi$ must virtually fix the homothety class of an  $(F_N,\calf)$-arational tree. By \Cref{Prop existence splitting arational tree}~$(1)$, the group $\phi$ fits in an exact sequence 
\[1 \to K'\to \langle \phi \rangle \to \ZZ \times \prod_{[A] \in \calf} \Out(A),\] where $K'$ consists of Dehn twists and the projection on the $\ZZ$ factor is nontrivial if and only if $\phi$ is fully irreducible relative to $\calf$.

For every $[A]\in \calf$, as $A$ is cyclic, the group $\Out(A)$ is finite. Thus, since $\phi$ is not a Dehn twist and has infinite order, the image of $\phi$ in the $\ZZ$ factor is nontrivial and $\phi$ is fully irreducible relative to $\calf$.

Thus, by \Cref{Theo loxo free factor}, $\phi$ is a loxodromic element of $\FF(F_N,\calf)$. Therefore, $C(\phi)$ fixes the attracting point of $\phi$ in $\partial_{\infty}\FF(F_N,\calf)$.

As above, by \Cref{Prop existence splitting arational tree}, since $\calf$ is nonsporadic, the group $C(\phi)$ fits into an exact sequence 
\[1 \to K \to C(\phi) \to \ZZ \times \prod_{[A] \in \calf} \Out(A)\] 
whose kernel $K$ is an abelian group. Since $\mathcal{F}$ is nonsporadic, for every $[A] \in \mathcal{F}$, the group $A$ is cyclic. Thus, for every $[A] \in \mathcal{F}$, the group $\Out(A)$ is isomorphic to $\ZZ/2\ZZ$. In particular, the image of the above homomorphism is virtually abelian. Thus, the group $C(\phi)$ does not contain a nonabelian free group. By  \Cref{Lem Tits}, the group $C(\phi)$ is abelian.
\end{proof}

\begin{lemma}\label{Lem:virtuallyfreeimpliesfree}
    Let $H$ be a subgroup of $\IA_3(3)$ preserving a free factor $A$ of $F_3$ of rank $2$. The image of $H$ in $\Out(A)$ is a free group.
\end{lemma}

\begin{proof}
    Since $A$ is a free factor of $F_3$, the image of $H$ in $\Out(A)$ is contained in $\IA_2(3)$. Since $A$ has rank $2$, the group $\Out(A)$ is isomorphic to $\mathrm{GL}_2(\ZZ)$ and is virtually free. Since $\IA_2(3)$ is torsion free by \Cref{Prop torsion free}, the group $\IA_2(3)$ is free and so is the image of $H$.
\end{proof}

\begin{lemma}\label{Lem:sporadiccasetwoelements}
Let $\phi \in \IA_3(3)$. Let $\mathcal{F}$ be a maximal $C(\phi)$-invariant free factor system. Suppose that $\mathcal{F}=\{[A],[B]\}$ with $F_3=A \ast B$ and $A$ is a nonabelian free group of rank $2$. Let $\Phi \in \phi$ be the unique representative of $\phi$ such that $\Phi(A)=A$ and $\Phi(B)=B$. One of the following holds.

\begin{enumerate}
\item The group $C(\phi)$ is isomorphic to $(F \cap \IA_3(3)) \times \ZZ$, where $F$ is a free subgroup of rank at most $2$.
\item The element $\phi$ is a Dehn twist. 
\end{enumerate}
\end{lemma}

\begin{proof}
By for instance~\cite{levitt2005}, for every $\psi \in C(\phi)$, there exists a unique representative $\Psi \in \psi$ such that $\Psi(A)=A$ and $\Psi(B)=B$. In particular, if $\psi \in C(\phi)$, then $\Psi$ commutes with $\Phi$.

Note that, since the rank of $A$ is equal to $2$, the rank of $B$ is equal to $1$. By  \Cref{Theo:freefactorconjclassfixed}~$(2)$, for every $\psi \in C(\phi)$, the automorphism $\Psi$ fixes $B$ elementwise. Thus, the homomorphism which sends $\psi \in C(\phi)$ to $\Psi|_A \in \Aut(A)$ is injective.

Suppose first that $\Phi|_A$ is inner: there exists $g \in A$ such that $\Phi|_A=\mathrm{ad}_g$. Then, $\phi$ is a Dehn twist.

Suppose now that $\Phi|_A$ is not inner. By \Cref{Lem:virtuallyfreeimpliesfree}, the image of $C(\phi)$ in $\Out(A)$ is free and contained in the centraliser of the image of $\phi$. Thus, the image of $C(\phi)$ in $\Out(A)$ is infinite cyclic, generated by the outer class of a root $\sqrt{\Phi|_A}$ of $\Phi|_A$. By  \Cref{Prop:powercommute}, every element of $C(\phi)$ commutes with $\sqrt{\Phi|_A}$.

Hence $C(\phi)$ is isomorphic to $K \times \langle\sqrt{\Phi|_A} \rangle$, where $K$ is contained in the subgroup of inner automorphisms of $A$. An inner automorphism $\mathrm{ad}|_g$ commutes with $\Phi|_A$ if and only if $\Phi|_A(g)=g$. Thus, $K$ is contained in $K_{\Phi}=\{\mathrm{ad}_g \in \Aut(A)\;|\; g \in \Fix(\Phi|_A)\}$. Conversely, any element of $K_{\Phi}$ extends to an automorphism of $F_3$ commuting with $\Phi$. Thus, $C(\phi)$ is isomorphic to $(K_{\phi} \cap \IA_3(3)) \times \langle\sqrt{\Phi|_A} \rangle$. Moreover, the free group $K_{\phi}$ has rank at most $2$ by the work of Bestvina and Handel~\cite{BesHan92}.
\end{proof}

\begin{lemma}\label{Lem:sporadiccaseoneelement}
Let $\phi \in \IA_3(3)$. Let $\mathcal{F}$ be a  maximal $C(\phi)$-invariant free factor system. Suppose that $\mathcal{F}=\{[A]\}$ with $F_3=A \ast \ZZ$ and let $t$ be a generator of the second factor. One of the following holds.

\begin{enumerate}
\item The group $C(\phi)$ is isomorphic to $((F_\ell \times F_r)\cap \IA_3(3)) \times \ZZ$, where $F_\ell$ and $F_r$ are two finitely generated free groups of rank at most $2$.

\item The element $\phi$ is a Dehn twist. 
\end{enumerate}
\end{lemma}

\begin{proof}
Recall that, by~\cite{levitt2005}, the kernel $K$ of $C(\phi) \to \Out(A)$ is isomorphic to the intersection of $\IA_3(3)$ with a direct product $F_{\ell} \times F_r$ of two free (maybe cyclic or trivial) normal subgroups of $C(\phi)$. Every element of $F_{\ell} \times F_r$ is a Dehn twist. Moreover, as explained in the last paragraph of the proof of Case~$1$ in the proof of  \Cref{thm:shortexactsequencecentraliser}, each factor of $F_\ell \times F_r$ corresponds to the fixed subgroup of an automorphism in the outer class $\phi$. The bound on the rank of the free groups then follows from the work of Bestvina and Handel~\cite{BesHan92}.

We may suppose that $\phi$ is not a Dehn twist, so that the image of $\phi$ in $\Out(A)$ is not trivial. By \Cref{Lem:virtuallyfreeimpliesfree}, as in the proof of \Cref{Lem:sporadiccasetwoelements}, the image of $C(\phi)$ in $\Out(A)$ is infinite cyclic, generated by a root of $\phi$ which commutes with every element of $C(\phi)$.

Combining the above two paragraphs, we see that $C(\phi)$ is isomorphic to $((F_\ell \times F_r)\cap \IA_3(3)) \times \ZZ$, where $F_\ell$ and $F_r$ are two finitely generated free groups of rank at most $2$.
\end{proof}

\begin{proof}[Proof of \Cref{Theo type VF}] Let $\phi \in \IA_3(3)$ and let $C(\phi)$ be the centraliser of $\phi$ in $\IA_3(3)$. If $\phi$ is a Dehn twist, its centraliser in $\Out(F_3)$ is of type $\mathsf{VF}$ by the work of Rodenhausen--Wade~\cite{rodenhausen2015centralisers} (see also the work of Andrew--Martino~\cite{AndrewMartino2022}). Moreover, the short exact sequence given in case~$(4)$ follows from \Cref{Thm:sesDehntwist}.

So suppose that $\phi$ is not a Dehn twist. Let $\mathcal{F}$ be a maximal $C(\phi)$-invariant free factor system. If $\mathcal{F}$ is nonsporadic, by  \Cref{Lem:typevfnonsporadic}, the group $C(\phi)$ is abelian. By  \Cref{Theo virtually abelian is abelian}, it is finitely generated. If $\mathcal{F}$ is sporadic, by \Cref{Lem:sporadiccasetwoelements,Lem:sporadiccaseoneelement}, the centraliser of $\phi$ satisfies one of Assertions~$(2),(3)$ of  \Cref{Theo type VF}. This concludes the proof.
\end{proof}

\section{Proper geometric dimension of Weyl groups of Dehn twists}\label{sec:GeomDimDehn}

In this section, we specify our study to the case of Dehn twists outer automorphisms. We will prove the following proposition.

\begin{prop}\label{Prop:gdDehntwists}
    Let $N\geq 2$ and let $\phi \in \IA_N(3)$ be a Dehn twist. The geometric dimension $\underline{\gd}(W(\phi))$ of the Weyl group of $\phi$ is bounded by $2N-4$.
\end{prop}

Recall that for a group $G$, the \emph{proper geometric dimension}, $\underline{\gd}(G)$, is defined to be the minimal $n\in\NN\cup\{\infty\}$ such that $G$ admits an $n$-dimensional model for $\underline{E}G=E_\FIN G$.

In order to prove \Cref{Prop:gdDehntwists}, we take advantage of the short exact sequence given by Theorem~\ref{Thm:sesDehntwist}. We need to understand more precisely the groups $\Out(G_v,\mathrm{Inc}_v)$ as defined in this theorem. We first recall a result, due to Meucci (see also the work of Day--Sale--Wade~\cite{day2021calculating}).

\begin{thm}\cite{meucci2011relative}\label{Thm:dimensionrelative}
    Let $\calf$ be a free factor system of $F_N$, let $\calf^\ZZ$ be the subset of $\calf$ consisting of the conjugacy classes of cyclic subgroups and let $\calf^{\ge 2}=\calf-\calf^\ZZ$. Then $$\underline{\gd}(\Out(F_N,\calf^{(t)}))=2N-2\sum_{[A]\in \calf^{\ge 2}}(\rank(A)-1)-2-|\calf^\ZZ|.$$

\end{thm}

The proof of the following proposition is implicit in \cite[Proof~of~theorem~3.1]{Lu00}, see also \cite[Theorem~2.3]{MPSS20}. 

\begin{prop}\label{prop:haefliger}
Let $G$ be a group. Let $f:G\to Q$ be a group homomorphism. Let $Y$ be a model for $\underline{E}Q$. Then
\[\underline{\gd}(G)\leq \max\{\underline{\gd}(G_{\sigma})+\dim(\sigma)| \ 
\sigma \text{ is a cell of } \ Y \},\]
where the stabilisers are taken with respect to the $G$-action on $Y$ induced by the projection.

In particular if $Q$ torsion free and $K$ is the kernel of $f$, then  $\underline{\gd}(G)\leq \underline{\gd}(Q)+\underline{\gd}(K)$.
\end{prop}

For a subgroup $H \subseteq \Out(F_N)$, we say that $H$ satisfies \emph{Property~$(CF)$} if, for every $\psi \in H$, every $\psi$-periodic conjugacy class of elements of $F_N$ is fixed. By Theorem~\ref{Theo:freefactorconjclassfixed}~$(2)$, any subgroup of $\IA_N(3)$ satisfies Property~$(CF)$. Groups with Property~$(CF)$ satisfies the following result.

\begin{lemma}\label{lem:PropertyCF}
    Let $H \subseteq \Out(F_N)$ be a subgroup satisfying Property~$(CF)$ and let $A \subseteq F_N$ be a malnormal subgroup preserved by $H$. The image of $p\colon H \to \Out(A)$ is torsion free and satisfies Property~$(CF)$.
\end{lemma}

\begin{proof}
    Let $\psi \in H$ and let $k \in \NN^*$ be such that $p(\psi)^k=\mathrm{id}$. Then every conjugacy class of elements of $A$ is fixed by $\psi$ by Property~$(CF)$. Thus, $\psi$ has a representative which acts as the identity on $A$ and $\psi=\mathrm{id}$. The fact that $p(H)$ satisfies Property~$(CF)$ is immediate.
\end{proof}

Let $\calc=\{[x_1],\ldots,[x_k]\}$ be a finite set of conjugacy classes of elements of $F_N$. We will denote by $\calf_\calc$ the minimal free factor system of $F_N$ such that, for every $i\in \{1,\ldots,k\}$, there exists $[A]\in \calf_\calc$ with $x_i \in A$. 

\begin{corollary}\label{coro:boundvcdMccool}
    Let $\calc$ be a finite set of conjugacy classes of elements of $F_N$. Let $H$ be a subgroup of $\Out(F_N,\calc)$ which satisfies Property~$(CF)$. The geometric dimension of $H$ is bounded by 
    \[\underline{\gd}(H) \leq 2N-2-|\calf_\calc|.\]
\end{corollary}

\begin{proof}
    Since $H$ satisfies Property~$(CF)$, it fixes $\calc$ elementwise. Hence $H$ fixes $\calf_\calc$ elementwise. Thus, we have a natural homomorphism 
    \[1 \to K \to H \to \prod_{[A]\in \calf_\calc^{\ge 2}} \Out(A)\] whose kernel $K$ is a subgroup of $\Out(F_N,\calf_\calc^{(t)})$. By \Cref{lem:PropertyCF}, the image of $H$ is torsion free. By~\Cref{prop:haefliger}, the geometric dimension of $H$ is bounded by 
    \[
    \underline{\gd}(H) \leq \underline{\gd}(\Out(F_N,\calf_\calc^{(t)}))+\sum_{[A]\in \calf_\calc^{\ge 2}} \underline{\gd}(\Out(A)).
    \] 
    By~\cite{CV86}, for every $[A]\in \calf_\calc^{\ge 2}$, we have $\underline{\gd}(\Out(A))=2\rank(A)-3$. Combining this with \Cref{Thm:dimensionrelative}, we obtain
    \[
    \underline{\gd}(H) \leq 2N+\sum_{[A]\in \calf^{\ge 2}}(2\cdot\rank(A)-2\cdot\rank(A)+2-3)-2-|\calf_\calc^\ZZ|.
    \] Therefore, we have \begin{align*}\underline{\gd}(H) &\leq 2N-2-|\calf_\calc^{\ge 2}|-|\calf_\calc^\ZZ|\\
    &=2N-2-|\calf_\calc|.\qedhere\end{align*}
\end{proof}

We need another theorem due to Shenitzer~\cite{shenitzer1955decomposition} and Swarup~\cite{swarup1986decompositions} (see also the work of Stallings~\cite{stallings1991foldings} and Bestvina-Feighn~\cite[Lemma~4.1]{bestvina1994outer}) in order to understand the proper geometric dimension of the Weyl group of a Dehn twist.

\begin{lemma}\cite{shenitzer1955decomposition,swarup1986decompositions}\label{Lem:unfoldingedgecyclicsplitting}
    Let $T$ be an $F_N$-tree whose edge stabilisers are infinite cyclic. There exists an oriented edge $e_+ \in E^+(F_N \backslash T)$ with origin $v \in V(F_N \backslash T)$ which satisfies:
    \begin{enumerate}
        \item the group $G_v$ splits as $G_v=A \ast G_{e_+}$ for some nontrivial subgroup $A \subseteq G_v$;
        \item for every oriented edge $e_+' \in E^+(F_N \backslash T)$ with origin $v$ and distinct from $e_+$, some conjugate of $G_{e_+'}$ is contained in $A$. 
    \end{enumerate}
\end{lemma}

We can now prove the key lemma in order to bound the geometric dimension of the Weyl group of a Dehn twist. Recall the definition of $T_\cala$ for a Dehn twist $\phi \in \IA_N(3)$. If $v\in V(F_N \backslash T_\cala)$, Recall the definition of $\IA_v(3)$ in \Cref{Thm:sesDehntwist}. 

\begin{lemma}\label{Lem:Bounddimensionimageweylgroup}
    Let $N \geq 2$ and let $\phi \in \IA_N(3)$ be a Dehn twist. Let $T_\cala$ be the associated JSJ tree. Then 
    \[\sum_{v \in V(F_N \backslash T_\cala)} \underline{\gd}(\IA_v(3)) \leq 2N-3-|E(F_N \backslash T_\cala)|.\]
\end{lemma}

\begin{proof}
    In order to simplify the notations, let $E=E(F_N \backslash T_\cala)$ and let $V=V(F_N \backslash T_\cala)$. For every $v\in V$, let $\calf_v=\calf_{\mathrm{Inc}_v}|_{G_v}$ and let $\calf=\coprod_{v\in V} \calf_v$. Note that for every $v\in V$, the free factor system $\calf_v$ is a free factor system of $G_v$ and not of $F_N$.

    We claim that $|\calf|\geq |E|+1$. The proof is by induction on $|E|\geq 1$. 
    
    Suppose that $|E|=1$ and let $e$ be the (unoriented) edge of $F_N \backslash T_\cala$. If $e$ has two distinct endpoints $v$ and $w$, then, since the stabiliser of $e$ is nontrivial, we have $|\calf_v|,|\calf_w|=1$ and $|\calf|\geq 2$.

    Suppose now that $e$ is a loop based at $v$. By \Cref{Lem:unfoldingedgecyclicsplitting}, one of the orientation of $e$, say $e_+$, is such that the group $G_v$ splits as $G_v=A \ast G_{e_+}$ where $A$ is a nontrivial subgroup of $G_v$. Since $e$ is a loop, we have $\mathrm{Inc}_v=\{[G_{e_+}],[G_{e_-}]\}=\{[G_{e_+}],[tG_{e_+}t^{-1}]\}$, where $t\in F_N-G_v$ and $tG_et^{-1} \subseteq A$. Since $G_v=A \ast G_{e_+}$, we have $\calf=\calf_v=\{[A'],[G_{e_+}]\}$, where $A'$ is the smallest free factor of $A$ containing $tG_et^{-1}$. Thus, we have $|\calf| \geq 2$. This proves the base case.

    Suppose that $|E|\geq 2$, let $e \in E$, let $v \in V$ be adjacent to $e$.  We also assume that, once chosen the orientation $e_+$ of $e$ such that $v$ is the origin of $e_+$, the oriented edge $e_+$ is the one given by \Cref{Lem:unfoldingedgecyclicsplitting}. Let $w$ be the other endpoint of $e_+$ (which is possibly equal to $v$). Let $T'$ be the tree obtained from $T_\cala$ by collapsing the orbit of the edge $e_+$. Let $E'=E(F_N \backslash T')$ and let $V'=V(F_N \backslash T')$. For every $x\in V'$, let $\calf_x'=\calf_{\mathrm{Inc}_x}|_{G_x}$ and let $\calf'=\coprod_{x\in V'} \calf_x$. For simplicity, if $v=w$, we set $\mathrm{Inc}_w=\varnothing$. By induction, we have $|\calf'|\geq |E'|+1=|E|$. 
    
    For every $x\in V$, let $\overline{x}$ be the image of $x$ in $F_N \backslash T'$. Note that, for every $x \in V-\{v,w\}$, we have $\calf_{\overline{x}}'=\calf_x$. 
    
     Let $A \subseteq G_v$ be such that $G_v=A\ast G_{e_+}$. By Lemma~\ref{Lem:unfoldingedgecyclicsplitting}, for every oriented edge $e_+'\in E^+$ with origin $v$ and distinct from $e_+$, a conjugate of $G_{e_+'}$ is contained in $A$. Thus, we have $\mathrm{Inc}_{\overline{v}}\subseteq (\mathrm{Inc}_v-\{[G_{e_+}]\}) \cup \mathrm{Inc}_w$. Hence we see that \[|\calf_{\overline{v}}|\leq |\calf_{\mathrm{Inc}_v-\{[G_{e_+}]\}}|_{G_v}|+|\calf_{w}|.\]

    Since $\calf_{\mathrm{Inc}_v-\{[G_{e_+}]\}}|_{G_v} \leq \{[A]\}$ by~\Cref{Lem:unfoldingedgecyclicsplitting}, we also have \[\calf_v=\calf_{\mathrm{Inc}_v-\{[G_{e_+}]\}}|_{G_v} \coprod \{[G_{e_+}]\},\] so that $|\calf_v|=|\calf_{\mathrm{Inc}_v-\{[G_{e_+}]\}}|_{G_v}|+1$. 
    
    Combining all the above remarks, we see that 
    \begin{align*}
    |\calf| \geq & \sum_{x \in V-\{v,w\}} |\calf_{\overline{x}}'|+|\calf_v|+|\calf_w| \\
    \geq & \sum_{x \in V-\{v,w\}} |\calf_{\overline{x}}'|+|\calf_{w}|+|\calf_{\mathrm{Inc}_v-\{[G_{e_+}]\}}|_{G_v}|+1 \\
    \geq & \; |\calf'|+1 \geq |E|+1.
    \end{align*}
This proves the claim.

By for instance~\cite[Fact~7.3]{lustig1999conjugacy}, we have \[N-1=\sum_{v\in V} (\rank(G_v)-1).\] 

Note that $C(\phi)$ satisfies Property~$(CF)$ by \Cref{Theo:freefactorconjclassfixed}. By \Cref{lem:PropertyCF}, for every $v\in V$, the group $\IA_v(3)$ also satisfies Property~$(CF)$. By \Cref{coro:boundvcdMccool}, for every $v\in V$, we have \[\underline{\gd}(\IA_v(3)) \leq 2\rank(G_v)-2-|\calf_v|.\]

Thus, we see that 
\begin{align*}
   \sum_{v \in V(F_N \backslash T_\cala)} \underline{\gd}(\IA_v(3)) &\leq    \sum_{v \in V(F_N \backslash T_\cala)} (2\cdot\rank(G_v)-2-|\calf_v|) \\
   &=  2 \left(\sum_{v \in V(F_N \backslash T_\cala)} (\rank(G_v)-1)\right)-|\calf| \\
   &=  2(N-1)-|\calf|\\
   &\leq 2N-3-|E|;
\end{align*}
where the last inequality follows from the above claim.
\end{proof}

\begin{proof}[Proof of \Cref{Prop:gdDehntwists}] Let $\phi \in \IA_N(3)$ be a Dehn twist. Let $T_\cala$ be the JSJ tree associated with $\phi$. By \Cref{Thm:sesDehntwist}, the group $C(\phi)$ fits in a short exact sequence 
\[1 \to K' \to C(\phi) \to \prod_{v\in V(F_N \backslash T_\cala)} \IA_v(3) \to 1, \] where $K'$ is a free abelian group of rank $|E(F_N \backslash T_\cala)|$ and, for every $v\in V(F_N \backslash T_\cala)$, the group $\IA_v(3)$ is a finite index subgroup of $\Out(G_v,\mathrm{Inc}_v)$. 

Moreover, the outer automorphism $\phi$ is in the kernel of this homomorphism. Thus, the Weyl group $W(\phi)$ of $\phi$ fits in a short exact sequence 
\[1 \to K \to W(\phi) \to \prod_{v\in V(F_N \backslash T_\cala)} \IA_v(3) \to 1,\] where $K$ is an abelian group of rank $|E(F_N \backslash T_\cala)|-1$. By \Cref{Theo:freefactorconjclassfixed}~$(2)$ and \Cref{lem:PropertyCF}, the group $\IA_v(3)$ is torsion free. Thus, by~\Cref{prop:haefliger} and \Cref{Lem:Bounddimensionimageweylgroup}, we have 
\begin{align*}
    \underline{\gd}(W(\phi)) &\leq  \underline{\gd}(K)+ \sum_{v\in V(F_N \backslash T_\cala)} \underline{\gd}(\IA_v(3)) \\
    &\leq  2N-3-|E(F_N \backslash T_\cala)|+|E(F_N \backslash T_\cala)|-1 \\
    &\leq  2N-4,
\end{align*}
which concludes the proof.
\end{proof}

\section{The virtually cyclic dimension of \texorpdfstring{$\IA_N(3)$}{IAN(3)}}
\label{sec:ThmA}

In this section we prove \Cref{thmx.gdIA3}.  Before proceeding to the proof we need some preliminary results.

Let $G$ be a group. Denote by $\mathcal{VC}_\infty$ the collection of infinite virtually cyclic subgroups of $G$. Consider $[\mathcal{VC}_\infty]$ the set of commensuration classes of $\mathcal{VC}_\infty$. Let $I$ be a set of representatives of conjugacy classes in $[\mathcal{VC}_\infty]$. 

\begin{thm}\label{LuckWeiermann}
Let $G$ be a group. Let $I$ be defined as above, and assume that for each $H\in I$ we have $N_G[H]=N_G(H)$. For each $H\in I$, choose models for $\underline{E}N_G(H)$ and $\underline{E}W_G(H)$, where $W_G(H)=N_G(H)/H$. Now consider the homotopy $G$-pushout:
	  \[ \begin{tikzcd} 
   \coprod_{H\in I} G\times_{N_G(H)}\underline{E}N_G(H) \arrow[r,"i"] \ar[d,"\coprod_{H\in I}Id_G\times_{N_G(H)}f_{H}"'] & \underline{E}G \arrow[d] \\ \coprod_{H\in I}G\times_{N_G(H)} \underline{E}W_G(H) \arrow[r] & X
        \end{tikzcd} \]
	  where $\underline{E}W_G(H)$ is viewed as an $N_G(H)$-CW-complex by restricting with the projection $N_G(H)\to W_G(H)$, the maps starting from the left upper corner are cellular and one of them is an inclusion of $G$-CW-complexes. Then $X$ is a model for $\underline{\underline{E}}G$.
\end{thm}
\begin{proof}
   Let $H\in I$. Consider the quotient projection $p:N_G(H)\to W_G(H)$. Let $\mathcal{F}in^*$ be the family of subgroups of $N_G(H)$ that have finite image in $W_G(H)$. Let $\calf[H]$ be the family of all finite subgroups of $N_G[H]$ and all virtually cyclic subgroups of $N_G[H]$ that are commensurable with $H$. Note that $\mathcal{F}in^*=\calf[H]$, and as a consequence every model for $\underline{E}W_G(H)$ is a model for $E_{\calf[H]}N_G(H)$. Now the statement follows from \cite[Theorem 2.3]{LW12}. 
\end{proof}

An immediate corollary of~\Cref{prop:haefliger} is the following.

\begin{corollary}\label{Cor_LuckWeiermann}
    Under the hypothesis of \Cref{LuckWeiermann}, we have the following
    \[\underline{\underline{\gd}}(G)\leq \max\{ \underline{\gd}(G)+1, \underline{\gd}(W_G(H))|H\in I \}.\]
\end{corollary}

\begin{prop}\label{thm.gdIA3} Let $H$ be an infinite cyclic subgroup of $\IA_3(3)$, then 
      $\underline{\gd}(W(H))=\underline{\gd}(\IA_3(3))-1=2$.
\end{prop}
\begin{proof}
Denote $G=\IA_3(3)$, and recall that $G$ is torsion free.  Then by \Cref{Theo type VF} we have four possibilities for $C(H)$ (and so does for $W(H)$), let us work out each of them using the same numeration as in the theorem.

    \medskip
    
\noindent\textbf{Case 1.}
 The centraliser $C(H)$ is $\Z^r$ with $r\leq 3$. Thus $W(C)$ is an abelian group of rank at most 2. Hence $\RR^{r-1}$ is a model for $\underline{E}W(H)$, in particular this model has dimension at most 2.\hfill$\blackdiamond$

     \medskip
    
\noindent\textbf{Case 2.}
 The centraliser $C(H)$ is isomorphic to $F\times \ZZ$ where $F$ is a finitely generated free group and $H$ lies inside the $\ZZ$ under the isomorphism.  The Weyl group $W(H)$ has the form $F\times C$ where $C$ is a finite cyclic group. Since proper classifying space models respect direct products we have $\underline{E}W(H)=\underline{E}F\times \underline{E}C=\underline{E}F$ and the latter can be taken to be a tree.\hfill$\blackdiamond$

     \medskip
    
\noindent\textbf{Case 3.}
    The centraliser $C(H)$ is isomorphic to a direct product $K\times \ZZ$ where $K$ is a finite index subgroup of a direct product of two finitely generated free groups and $H$ lies in the $\ZZ$ factor under the isomorphism. In this case $C(H)\leq F_1\times F_2\times \ZZ$, hence the Weyl group $W(H)$ embeds into $F_1\times F_2\times C$ with $C$ a finite cyclic group. Thus it is enough to find a model for $\underline{E}(F_1\times F_2\times C)$ since by restriction it will be also a model for $\underline{E}W(H)$. As in the previous item such a model can be taken to be of the form $T_1\times T_2$ with $T_1$ and $T_2$ are trees.  Thus  we have a model for $\underline{E}W(H)$ of dimension 2.\hfill$\blackdiamond$

\medskip
    
\noindent\textbf{Case 4.}
    The group $H$ is generated by a Dehn twist. By~\Cref{Prop:gdDehntwists}, we have a model for $\underline{E}W(H)$ of dimension $2$.\hfill$\blackdiamond$

We have exhausted all possible cases for $W(H)$, completing the proof.
\end{proof}

\begin{prop}\label{prop:WeylN}
    Let $N\geq 3$. Consider a maximal infinite cyclic subgroup $H$ of $G=\IA_N(3)$. Then $\underline{\gd}(W_G(H))\leq \underline{\gd}(G)-1=2N-4.$
\end{prop}
\begin{proof}
    Let $\phi$ generate $H$.  We proceed by induction. Our base case $N=3$ was already proved in \Cref{thm.gdIA3}.

    Assume now that for every $N'< N$ the statement is true. Let $\calf=\{[A_1],\ldots,[A_\ell]\}$ be the free factor system given by Corollary~\ref{Coro:sesWeylgroup}. For every $i\in \{1,\ldots,\ell\}$, identify $\IA(A,3)$ with $\IA_{n_i}(3)$ where $\rank(A)=n_i$. Note that, for every $i\in \{1,\ldots,\ell\}$, we have $\underline{\gd}(\IA_{n_i}(3))=2n_i-3$ if $n_i \geq 2$ and $0$ otherwise. Now, we will prove the statement for $N$, by exhausting all cases for $W(H)$ described in \Cref{Coro:sesWeylgroup}. We use the notation from that theorem without further explanation. 
    
\medskip

\noindent\textbf{Case 1.} 
$\phi$ is a Dehn twist.  This is already proven in \Cref{Prop:gdDehntwists}.\hfill$\blackdiamond$

\medskip

\noindent\textbf{Case 2.} 
The group $W(\phi)$ is isomorphic to a subgroup of $\Out(A) \times \Out(B)$, where $A,B \subseteq F_N$ are such that $\rank(A)+\rank(B)=N+1$ and $\rank(A), \rank(B) \leq N-1$.  

We have $\underline{\gd}(\Out(A))=2\cdot\rank(A)-3$ and $\underline{\gd}(\Out(B))=2\cdot\rank(B)-3$ (both $A$ and $B$ are non-abelian). So 
\begin{align*}
\underline{\gd}(W(\phi))&\leq 2(\rank(A)+\rank(B))-6\\
&=2(N+1)-6\\
&=2N-4 
\end{align*}
as required.\hfill$\blackdiamond$

\medskip

\noindent\textbf{Case 3.} 
The group $W(\phi)$ is isomorphic to a subgroup of $\Out(A \ast \langle s^t\rangle, [s], [s^t])$, where $A \subseteq F_N$, $\rank(A)=N-1$, $s \in A$ and $t$ is a basis element of $F_N$. Let $\calc=\{[s],[s^t]\}$ and note that $|\calf_\calc|=2$ since $s \in A$ (see the notations in \Cref{coro:boundvcdMccool}). By \Cref{coro:boundvcdMccool}, a model for $\underline{E}\Out(A \ast \langle s^t\rangle, [s], [s^t])$ has dimension $$2\rank(A \ast \langle s^t\rangle)-2-|\calf_\calc|=2N-4.$$

A model for $\underline{E}W(\phi)$ is given by $\underline{E}\Out(A \ast \langle s^t\rangle, [s], [s^t])$. Thus, 
\[\underline{\gd}(W(\phi)) \leq 2N-4\]
as required.\hfill$\blackdiamond$

\medskip

\noindent\textbf{Case 4.} 
The group $W(\phi)$ is isomorphic to a subgroup of $\Aut(A) \times \Aut(B)$, where $A,B \subseteq F_N$ are nontrivial subgroups such that $F_N=A \ast B$.

If $\rank(A)=1$, then $\Aut(A)$ is finite and so $\underline{\gd}(\Aut(A))=0$.  Otherwise, a model $X$ for $\underline{E}\Aut(A)$ is given by an $\Aut(A)$-fibration
$X\to \underline{E}\Out(A)$.  Here, the stabilisers of the $\Aut(A)$-action on $\underline E\Out(A)$ are virtually free, being extensions of $A$ by finite subgroups of $\Out(A)$.  It follows from \Cref{prop:haefliger} that $\underline{\gd}(\Aut(A))=2\rank(A)-2$.  An identical argument gives $\underline{\gd}(\Aut(B))\leq 2\rank(B)-2$.  Now, a model for for $\underline{E}W(\phi)$ is given by $\underline{E}\Aut(A)\times \underline{E}\Aut(B)$.  So,
\begin{align*}
    \underline{\gd}(W(\phi)) &\leq 2(\rank(A)+\rank(B))-4\\
        &=2N-4
\end{align*}
as required.\hfill$\blackdiamond$

\medskip

\noindent\textbf{Case 5.}
 The Weyl group $W(\phi)$ fits into an exact sequence 
        \[1 \to K \to W(\phi) \to H_\infty/\rho_\infty(\langle \phi \rangle) \times H_T,\]
        where $K$ is a finite index subgroup of a direct product of two finitely generated free groups.

        In this case $\calf=\{A_1,A_2\}$, $A_1$ is never trivial, and $A_2$ might be trivial or not, see \Cref{Rmk:rankfreefactorsystem}.  We subdivide into two subcases depending on whether $A_2$ is trivial or not.

        Assume $A_2$ is trivial. Then $\calf=\calf_\infty$,  $\calf_T$ is empty, and $n_1\leq N-1$. Hence, we have the following exact sequence 
        \[1 \to K \to W(\phi) \to C(\langle\phi|_{A_1}\rangle)/\rho_\infty(\langle \phi \rangle)=W_{\IA(A_1,3)}(\rho_\infty(\langle \phi \rangle)).\]
        Now, by hypothesis $\phi$ is not a proper power.  This implies that $W_{\IA(A_1,3)}(\rho_\infty(\langle \phi \rangle))$ is torsion free. Now by \Cref{prop:haefliger} and the induction hypothesis we get
        \begin{align*}
            \underline{\gd}(W(\phi)) & \leq 2+\underline{\gd}(W_{IA(A_1,3)}(\rho_\infty(\langle \phi \rangle)))\\
            &\leq 2+2n_1-4\\
            &\leq 2+2(N-1)-4=2N-4.
        \end{align*}

        Now, assume $A_2$ is not trivial. We distinguish two cases depending on $\calf_\infty$. If $\calf_\infty=\{A_1\}$, then we have the exact sequence 
         \[1 \to K \to W(\phi) \to W_{IA(A_1,3)}(\rho_\infty(\langle \phi \rangle))\times \IA(A_2,3).\]
        By \Cref{prop:haefliger} and the induction hypothesis we get
        \begin{align*}
            \underline{\gd}(W(\phi)) & \leq 2+\underline{\gd}(W_{IA(A_1,3)}(\rho_\infty(\langle \phi \rangle))) + \underline{\gd}(\IA(A_2,3))\\
            &\leq 2+2n_1-4+2n_2-3\\
            &\leq 2(n_1+n_2)-5\leq 2N-4,
        \end{align*}
        where the last inequality follows from the fact that $\sum n_i\leq N$, see \Cref{Rmk:rankfreefactorsystem}.
        Finally,  if $\calf_\infty=\{A_1,A_2\}$, then we have the exact sequence 
         \[1 \to K \to W(\phi) \to (C(\phi|_{A_1})\times C(\phi|_{A_2}))/\rho_\infty(\langle \phi \rangle)=:Q.\]
         On the other hand, taking the quotient of $C(\phi|_{A_1})\times C(\phi|_{A_2})$ by $\langle \phi|_{A_1}, \phi|_{A_2}\rangle\cong \ZZ^2$ , we conclude $Q$ fits in the following exact sequence
         \[1\to \ZZ \to Q\to W(\phi|_{A_1})\times W(\phi|_{A_2}). \]
         Hence, we get
         \begin{align*}
             \underline{\gd}(Q) &\leq 1+ \underline{\gd}(W(\phi|_{A_1}))+ \underline{\gd}(W(\phi|_{A_2}))\\
             &\leq 1+2n_1-4+n_2-4\\
             &\leq 2N-7
         \end{align*}
         and therefore
         \[\underline{\gd}(W(\phi)\leq 2+2N-7=2N-5\]
         as required.

        \medskip

        \noindent\textbf{Case 6.}
        The Weyl group $W(\phi)$ fits into an exact sequence
        \[ 1 \to K \to W(\phi) \to \prod_{[A]\in\calf}\IA(A,3)\]
        where $K$ is the direct product of a finitely generated free (maybe cyclic or trivial) group and a finite group. Moreover, we have $|\calf|\leq 2$ by \Cref{Thm:sesoneendedcase}~(2). Write $K=K'\times T$ where $K'$ is a (possibly trivial) free group and $T$ is the finite group. Thus, we may rewrite the quotient as $ \IA_{n_1}(3) \times \IA_{n_2}(3)$ such that each $n_i\leq N-1$ and $n_1+n_2\leq N$ (see~Remark~\ref{Rmk:rankfreefactorsystem}).  We may now build a model $X$ for $\underline{E}W(\phi)$ as a $G$-fibration
        \[ \underline EK' \to X \to  E\IA_{n_1}(3) \times E\IA_{n_2}(3). \]
        Since $K'$ is free we have $\underline{\gd}(K')=1$ and on the other end of the fibration we have $\underline{\gd}(\IA_{n_i}(3))=2n_i-3$ if $n_i \geq 2$ and $0$ otherwise.  If both $n_1,n_2 \geq 2$, then 
        \[\underline{\gd}(W(\phi))\leq 1+\sum_{i=1}^2 (2n_i-3)\leq 1+2N-6\leq 2N-4.\] 
        If say $n_2 \leq 1$, then $2\leq n_1 \leq N-1$ and 
        \[\underline{\gd}(W(\phi))\leq 1+2(N-1)-3)\leq 1+2N-5\leq 2N-4\]
        as required.\hfill$\blackdiamond$

    We have exhausted all possible cases for $W(\phi)$ completing the proof.
\end{proof}

\begin{thmx}\label{thmx.gdIA3}
    Let $N\geq 1$.  Then, $\underline{\underline{\gd}}(\IA_N(3)))=2N-2$.
\end{thmx}
\begin{proof}
    The statement is clear for $N=1$. For $N=2$ follows from the fact that $\IA_2(N)$ is nonabelian free. Let $N\geq 3$. By \cite{CV86} $\Out(F_N)$ admits a model for $\underline{E}\Out(F_N)$ of dimension $2N-3$, thus $\underline{\gd}(\IA_N(3))\leq \underline{\gd}(\Out(F_N))\leq 2N-3$. From the latter inequality, \Cref{Cor_LuckWeiermann} and \Cref{prop:WeylN} we conclude $\underline{\underline{\gd}}(\IA_N(3))\leq 2N-2$. On the other hand, by \cite{CV86} there is a copy of $\ZZ^{2N-3}$ inside $\IA_N(3)$. By \cite{CFH06} we have $\underline{\underline{\gd}}(\ZZ^{2N-3})=2N-2$. Therefore $2N-2\leq \underline{\underline{\gd}}(\IA_N(3))$. This concludes the proof.
\end{proof}

\begin{corx}\label{corx.gdOutFn}
    Let $N\geq 1$.  Then, $\underline{\underline{\gd}}(\Out(F_N))$ is finite.
\end{corx}
\begin{proof}
    This is trivial for $N=1$.  Suppose now $N\geq2$.  By \cite[Theorem~2.4]{Lu00} there exists a model for $\underline{\underline{E}}\Out(F_N)$ which has dimension at most $|\Out(F_n)\colon\IA_N(3)|\cdot\underline{\underline{\gd}}(\IA_N(3))$
    which is finite.
\end{proof}

Note that this gives the upper bound
\[\underline{\underline{\gd}}(\Out(F_N))\leq |\GL_N(3)|\cdot \underline{\underline{\gd}}(\IA_N(3))=3^{\frac{1}{2}N(N-1)}\cdot(2N-2)\cdot\prod_{i=1}^{N-1}(3^{N-i}-1).\]

\begin{question}
    Is $\underline{\underline{\gd}}(\Out(F_N))=2N-2$?
\end{question}

\bibliographystyle{alpha}
\bibliography{ref.bib}

\end{document}

%% file: top.tex
\newcounter{commentcounter}

\usepackage{amsmath} 
\usepackage{amssymb} 
\usepackage{amsthm} 
\usepackage{stmaryrd} 
\usepackage[english]{babel} 
\usepackage[font=small,justification=centering]{caption} 
\usepackage[nodayofweek]{datetime}
\usepackage[shortlabels]{enumitem} 
\usepackage[T1]{fontenc} 
\usepackage[utf8]{inputenc} 
\usepackage{ifthen} 
\usepackage{mathabx} 
\usepackage{mathtools} 
\usepackage[dvipsnames]{xcolor} 
\usepackage[pdftex,  colorlinks=true,  pagebackref=true]{hyperref} 
    \hypersetup{urlcolor=RoyalBlue, linkcolor=RoyalBlue,  citecolor=black}
\usepackage{tikz-cd} 
\usepackage{xfrac} 
\usepackage[capitalize]{cleveref} 

\renewcommand*{\backref}[1]{}
\renewcommand*{\backrefalt}[4]
{
    \ifcase #1
        No citation in the text.
    \or
        Cited on page #2.
    \else
        Cited on pages #2.
    \fi
}

\makeatletter
\@namedef{subjclassname@1991}{Mathematical subject classification 1991}
\@namedef{subjclassname@2000}{Mathematical subject classification 2000}
\@namedef{subjclassname@2010}{Mathematical subject classification 2010}
\@namedef{subjclassname@2020}{Mathematical subject classification 2020}
\makeatother

\newtheorem{thm}{Theorem}[section]
\newtheorem{lemma}[thm]{Lemma}
\newtheorem{corollary}[thm]{Corollary}
\newtheorem{prop}[thm]{Proposition}

\newtheorem{question}[thm]{Question}

\newtheorem{thmx}{Theorem}
\newtheorem{corx}[thmx]{Corollary}

\theoremstyle{definition}

\newtheorem{remark}[thm]{Remark}

\theoremstyle{plain}

    \newtheoremstyle{TheoremNum}
        {\topsep}{\topsep} 
        {\itshape} 
        {-0.25cm} 
        {\bfseries} 
        {.} 
        { }  
        {\thmname{#1}\thmnote{ \bfseries #3}}
    \theoremstyle{TheoremNum}
    \newtheorem{duplicate}{}

\newcommand*{\claimproofname}{My proof}

\DeclareMathOperator{\Aut}{\mathrm{Aut}}
\DeclareMathOperator{\Out}{\mathrm{Out}}

\DeclareMathOperator{\Fix}{\mathrm{Fix}}
\DeclareMathOperator{\Stab}{\mathrm{Stab}}

\DeclareMathOperator{\Isom}{\mathrm{Isom}}
\DeclareMathOperator{\MCG}{\mathrm{MCG}}
\DeclareMathOperator{\IA}{\mathrm{IA}}

\newcommand{\cala}{{\mathcal{A}}}

\newcommand{\calc}{{\mathcal{C}}}

\newcommand{\calf}{{\mathcal{F}}}
\newcommand{\calg}{{\mathcal{G}}}

\newcommand{\calt}{{\mathcal{T}}}

\newcommand{\FIN}{\mathcal{FIN}}

\newcommand{\GL}{\mathrm{GL}}

\newcommand{\CAT}{\mathrm{CAT}}

\newcommand{\rank}{\mathrm{rk}}
\newcommand{\gd}{\mathrm{gd}}

\def\Z{\mathbb{Z}}

\newcommand{\NN}{\mathbb{N}}
\newcommand{\ZZ}{\mathbb{Z}}

\newcommand{\RR}{\mathbb{R}}

\newcommand{\FF}{\mathbb{F}}

\usepackage{tikz}
\usetikzlibrary{arrows,quotes}
\tikzstyle{blackNode}=[fill=black, draw=black, shape=circle]